\documentclass[letter,12pt]{article}

\usepackage{hyperref}
\usepackage[english]{babel}
\usepackage{amssymb}
\usepackage{amsthm}
\usepackage{amsfonts}
\usepackage{amsmath}
\usepackage{anysize}
\usepackage{bbm}
\usepackage{abstract}
\marginsize{2.6cm}{2.6cm}{1cm}{2cm}
\usepackage[T1]{fontenc}
\usepackage{upquote}

\newtheorem{theorem}{Theorem}[section]
\newtheorem{lemma}[theorem]{Lemma}
\newtheorem{prop}[theorem]{Proposition}
\newtheorem{coro}[theorem]{Corollary}

\theoremstyle{definition}

\newtheorem{remark}[theorem]{Remark}

\newcommand{\NN}{\mathbb{N}}
\newcommand{\RR}{\mathbb{R}}
\newcommand{\PP}{\mathbb{P}}

\newcommand{\CC}{\mathbb{C}}

\newcommand{\Cc}{\mathcal{C}}
\newcommand{\Ec}{\mathcal{E}}

\newcommand{\BB}{\mathbb{B}}

\newcommand{\setdef}{\ \vert \ }
\newcommand{\vep}{\varepsilon}
\newcommand{\psh}{{\rm PSH}}

\newcommand{\capa}{{\rm Cap}}
\newcommand{\capo}{{\rm Cap}_{\omega}}

\newcommand{\vol}{{\rm Vol}}

\newcommand{\AM}{{\rm I}}

\newcommand{\Amp}{{\rm Amp}}

\title{On the singularity type of full mass currents in big cohomology classes}

\author{Tam\'as Darvas, Eleonora Di Nezza, Chinh H. Lu}
\date{\vspace{-0.5cm}}

\begin{document}

\maketitle

\begin{abstract}Let $X$ be a compact K\"ahler manifold and $\{\theta\}$ be a big cohomology class. We prove several results about the singularity type of full mass currents, answering a number of open questions in the field. First, we show that the Lelong numbers and multiplier ideal sheaves of $\theta$-plurisubharmonic functions with full mass are the same as  those of the current with minimal singularities. Second, given another big and nef class $\{\eta\}$, we show the inclusion $\Ec(X,\eta) \cap \psh(X,\theta) \subset \Ec(X,\theta).$ Third, we characterize big classes whose full mass currents are ``additive''.  Our techniques make use of a characterization of full mass currents in terms of the envelope of their singularity type. As an essential ingredient we also develop the theory of weak geodesics in big cohomology classes.  Numerous applications of our results to complex geometry are also given. 
\end{abstract}

\section{Introduction and main results}

Since the fundamental work of Aubin \cite{Aub} and Yau \cite{Yau}, the complex Monge-Amp\`ere operator has found many important applications in differential geometry. In this vast area of research, pluripotential theory plays a crucial role, initiated by the seminal work of Bedford--Taylor \cite{BT76,BT82,BT87} and Ko{\l}odziej \cite{Kol98}, to only mention a few. 

Guedj and Zeriahi extended Bedford--Taylor theory to compact K\"ahler manifolds $(X,\omega)$  \cite{GZ05, GZ07}. Their idea was to extend the definition of the complex Monge--Amp\`ere operator to much larger sets of potentials, not only bounded ones. As a result, an adequate variational theory could be devised for global equations of complex Monge--Amp\`ere type \cite{BBGZ13,BBEGZ12} that has found many striking applications in K\"ahler geometry. 

Additionally, the methods of \cite{GZ07} have proven to be very robust, as they also apply in case of big cohomology classes that are non--K\"ahler, as explored in \cite{BEGZ10}. Given a smooth $(1,1)$-from $\theta$ on $X$, we say that the class $\{\theta\}$ is \emph{big}, if there exists a quasi--plurisubharmonic function $u$ on $X$ such that $\theta + dd^c u \geq \varepsilon \omega$  for some $\varepsilon >0$. 
Non--K\"ahler big classes arise naturally in constructions of algebraic geometry. Given a one point blowup of an arbitrary K\"ahler manifold, the simplest such example is given by the sum of the exceptional divisor class and a ``very small''  K\"ahler class. 

When varying K\"ahler classes, one often has to study degenerate classes as well, and there has been a lot of work in trying to characterize the degenerate classes that admit special K\"ahler metrics \cite{BBEGZ12,Dar16,SSY}.  
Recently, the solution to the complex Monge--Amp\`ere equation in a big class has been  used to show that the cone of pseudoeffective classes is dual to the movable cone, solving an important open problem in complex algebraic geometry \cite{WN16}. 

In most of the above mentioned works that study degenerate metrics finite energy pluripotential theory plays an important role. Partly motivated by this, and partly by a survey of open questions \cite{DGZ15}, we will investigate further the finite energy pluripotential theory of big cohomology classes. For a big class $\{ \theta\}$, the class of full mass currents $\mathcal E(X,\theta)$ is of central interest, as in many ways it is the analog of the classical Sobolev spaces, given the role it plays in the variational study of complex Monge-Amp\`ere equations (for the precise definition, see Section 2.1). Our first main result clarifies the local/global singular behavior of potentials in $\mathcal E(X,\theta)$ in various settings of geometric interest:  

\begin{theorem}\label{thm: lelong number big class}
        Let $(X,\omega)$ be a K\"ahler manifold. Assume that $\theta$ is a smooth closed $(1,1)$-form such that $\{\theta\}$ is big. Let $V_{\theta}$ be the envelope of $\theta$. Then we have the following: \\
\noindent (i) for any $\varphi\in \mathcal{E}(X,\theta)$ we have
$$\nu(\varphi,x) = \nu(V_{\theta},x) \textup{ and }\mathcal I(t\varphi,x)=\mathcal I(tV_\theta,x),  \ \forall x\in X, t >0,$$ where $\nu(\varphi,x)$ is the Lelong number of $\varphi$ at $x$, and $\mathcal I(t\varphi,x)$ is the germ of the multiplier ideal sheaf of $t\varphi$ at $x$. \\
\noindent (ii) If $\{\eta\}$ is a big and nef class, then $$\Ec(X,\eta) \cap \psh(X,\theta) \subset \Ec(X,\theta).$$  
In particular, when $\theta = \omega$, this last inclusion gives that $\nu(\varphi,x)=0$ for any $x \in X, \varphi \in \Ec(X,\eta)$.
\end{theorem}
Here, $V_\theta$ is the ``least singular'' element of ${\rm PSH}(X,\theta)$, and for the precise definition of all concepts in the above result we refer to Section 2.1. The statement of (ii) cannot hold in case $\left\{ \eta \right\}$ is merely big.  Indeed, if $\{\eta\}$ is big but not nef, the envelope $V_{\eta}$ may have a positive Lelong number at some point $x \in X$, hence its complex Monge-Amp\`ere measure can not have full mass with respect to a K\"ahler form $\omega\geq \eta$, as shown in \cite[Corollary 1.8]{GZ07}.  

In the particular case when $\{\theta\}$ is semi-positive and big,  Theorem \ref{thm: lelong number big class} answers affirmatively an open question in \cite[Question 36]{DGZ15}, saying that potentials in $\mathcal E(X,\theta)$ have zero Lelong numbers. A very specific instance of this was verified in \cite[Theorem 1.1]{BBEGZ12}, using techniques from algebraic geometry. 

 Our arguments use the envelope construction originally due to Ross--Witt Nystr\"om \cite{RWN} that we recall now. For  an upper semi-continuous function $f$ on $X$, we let $P_{\theta}(f)$ be the largest $\theta$-psh function lying below $f$, i.e.
$P_{\theta}(f)=\sup \{u\; |\; \theta {\rm-psh}\; u\leq f\}$.  
Given $\psi,\varphi$, two $\theta$-psh functions, we define:
\begin{equation}\label{eq: RWN envelope}
P_{[\theta,\psi]}(\varphi):= \Big\{ \lim_{C\to +\infty} P_{\theta}(\min(\psi+C,\varphi))\Big\}^*.
\end{equation}
In Section \ref{sect: preliminaries},  we prove that whenever $\varphi,\psi$ belong to $\mathcal{E}(X,\theta)$ then  $P_{\theta}(\min(\varphi,\psi))$ also belongs to $\mathcal{E}(X,\theta)$. Coincidentally, with the help of this result we can settle a conjecture in \cite[Remark 2.16]{BEGZ10}, regarding the convexity of finite energy classes associated to a big cohomology class (Corollary \ref{cor: BBEGZ_conj}).

To familiarize the reader with the flavor of our arguments, we sketch the proof of the statement involving Lelong numbers in Theorem \ref{thm: lelong number big class}(i) in the semi-positive case (when $\theta \geq 0$). If $\varphi\in \mathcal{E}(X,\theta)$ it follows from an approximation and balayage argument \cite{BT82} that the Monge--Amp\`ere measure of $P_{[\theta,\varphi]}(0)$ vanishes on $\{P_{[\theta,\varphi]}(0)<0\}$. It thus follows from the domination principle that $P_{[\theta,\varphi]}(0)=0$. For a K\"ahler form $\omega>\theta$, we have that $P_{[\omega,\varphi]}(0)=0$, since $P_{[\omega,\varphi]}(0)\geq P_{[\theta,\varphi]}(0)=0$. Hence \cite[Theorem 3]{Dar13} yields that $\varphi\in \mathcal{E}(X,\omega)$. This together with \cite[Corollary 1.8]{GZ07} imply that $\varphi$ has zero Lelong numbers everywhere on $X$.

The simple argument described above relies on  a surprising  characterization of the class $\mathcal{E}(X,\omega)$ in terms of the  envelope construction of  \eqref{eq: RWN envelope}  \cite[Theorem 3]{Dar13},\cite[Theorem 4]{Dar14a}. Our next result, which is a vital ingredient in our proof of Theorem \ref{thm: lelong number big class}, shows that this characterization holds in the context of big classes as well:  
\begin{theorem}\label{thm: characterization E big classes}
Let $\{\theta\}$ be a  big cohomology class and fix $\varphi\in \mathcal{E}(X,\theta)$. Then a function $\psi\in {\rm PSH}(X,\theta)$ belongs to $\mathcal{E}(X,\theta)$ if and only if $P_{[\theta,\psi]}(\varphi)=\varphi$.
\end{theorem}
In order to prove the above result, we introduce the seemingly unrelated notion of weak geodesics in big cohomology classes, mimicking Berndtsson's construction in the K\"ahler case \cite[Section 2.2]{Bern}, and we prove that the Monge--Amp\`ere energy $\AM$ (sometimes called Aubin--Yau or Aubin--Mabuchi energy) is  convex/linear  along weak  subgeodesics/geodesics (Theorems \ref{thm: AM convex big} and \ref{thm: AM linear big}). Compared to the K\"ahler case, this is a very subtle issue and it serves as the key technical ingredient  in the proof of Theorem \ref{thm: characterization E big classes}.

When varying big classes, an important question is to understand how the class of full mass currents changes. Theorem \ref{thm: lelong number big class}(ii) already establishes a result in this direction in the particular case of big and nef classes. Paralleling this, as a consequence of Theorem \ref{thm: characterization E big classes},  we can characterize the pairs of big classes that have ``additive'' full mass currents, greatly generalizing \cite[Theorem B]{DN15} in the process: 
\begin{theorem}\label{thm: sum_FULL_MASS_1}
Let $\{ \theta_1\}, \{ \theta_2\}$ be big classes on $X$. The following are equivalent:\\
(i) $ V_{\theta_1}+V_{\theta_2} \in \mathcal E(X, \theta_1 + \theta_2).$\\
(ii) For any $u \in \psh(X,\theta_1), v \in \psh(X,\theta_2)$ we have
\[
u+v \in \mathcal E(X,\theta_1 + \theta_2) \Longleftrightarrow u\in \Ec(X,\theta_1), v\in \Ec(X,\theta_2).
\]      
\end{theorem}
As it turns out, when $\{\theta_1\},\{\theta_2\}$ are big and nef, condition (i) in the above theorem is automatically satisfied (Corollary \ref{thm: sum_FULL_MASS_nef}). 
This result also helps to partially confirm conjecture \cite[Conjecture 1.23]{BEGZ10} concerning log concavity of the non-pluripolar  complex Monge-Amp\`ere measure in the case of full mass currents of big and nef classes (see Corollary \ref{cor: log concave big class}).

\paragraph{Organization of the paper.}

Section \ref{sect: preliminaries} mostly reviews background material on the pluripotential theory of big cohomology classes and we establish some preliminary results. In Section \ref{sect: weak geodesic} we develop the theory of weak geodesics in big cohomology classes following Berndtsson's ideas, and then we prove Theorem \ref{thm: characterization E big classes} and Theorem \ref{thm: sum_FULL_MASS_1}. Theorem \ref{thm: lelong number big class} will be proved in Section \ref{sect: proof of Theorem 1.1} while some other applications will be discussed in Section \ref{sect: application}.

\section{Pluripotential theory in big cohomology classes}\label{sect: preliminaries}

\subsection{Non-pluripolar Monge-Amp\`ere measures}
We recall basic facts concerning pluripotential theory  of big cohomology classes. We borrow notation and terminology from \cite{BEGZ10}, and we also refer to this work for further details. 

Let $(X,\omega)$ be a compact K\"ahler manifold of dimension $n$. We fix $\theta$ a smooth closed $(1,1)$-form on $X$ such that $\{\theta\}$ is \emph{big}, i.e., there exists $\psi \in {\rm PSH}(X,\theta)$ such that $\theta +dd^c \psi \geq \vep \omega$ for some small constant $\vep>0$. Here, $d$ and $d^c$ are real differential operators defined as $d:=\partial +\bar{\partial},\,  d^c:=\frac{i}{2\pi}\left(\bar{\partial}-\partial \right).$ A function $\varphi: X\rightarrow \mathbb{R}\cup\{-\infty\}$ is called quasi-plurisubharmonic  if it is locally written as the sum of a plurisubharmonic function and a smooth function. $\varphi$ is called $\theta$-plurisubharmonic  ($\theta$\emph{-psh}) if it is quasi-psh such that $\theta+dd^c \varphi\geq 0$ in the sense of currents. We let $\psh(X,\theta)$ denote the set of $\theta$-psh functions which are not identically $-\infty$ (equivalently, it consists of $\theta$-psh functions which are integrable on $X$).

A $\theta$-psh function $\varphi$ is said to have \emph{analytic singularities} if there exists $c>0$ such that locally on $X$,
$$
\varphi=\frac{c}{2}\log\sum_{j=1}^{N}|f_j|^2+u,
$$
where $u$ is smooth and $f_1,\dots,f_N$ are local holomorphic functions. The \emph{ample locus} $\Amp(\{\theta\})$ of $\{ \theta\}$ is the set of points $x\in X$ such that there exists a K\"ahler current $T\in \{ \theta \}$ with analytic singularities and smooth in a neighbourhood of $x$. The ample locus $\Amp(\{\theta\})$ is a Zariski open subset, and it is nonempty \cite{Bou04}.

Let $x \in X$. Fixing a holomorphic chart $x \in U \subset X$, the \emph{Lelong number} $\nu(\varphi,x)$ of $\varphi \in \textup{PSH}(X,\theta)$ is defined as follows:
\begin{equation}\label{eq: Lelongdef}
\nu(\varphi,x) = \sup\{ \gamma \geq 0 \textup{ s.t. } \varphi(z) \leq \gamma \log \|z-x \| + O(1) \textup{ on } U\}.
\end{equation}
One can also associate to $\varphi$ a collection of \emph{multiplier ideal sheafs} $\mathcal I (t\varphi), \ t \geq 0,$ whose germs are defined by:

\begin{equation}\label{eq: MultIdealdef}
\mathcal I(t\varphi,x) = \{ f \in \mathcal O_x \textup{ s.t. } \int_V |f|e^{-t\varphi} \omega^n< \infty \textup{ for some open set } x\in V \subset X\}.
\end{equation}

If $\varphi$ and $\varphi'$ are two $\theta$-psh functions on $X$, then $\varphi'$ is said to be \emph{less singular} than $\varphi$ if they satisfy $\varphi\le\varphi'+C$ for some $C\in \mathbb{R}$. A  $\theta$-psh function $\varphi$ is said to have \emph{minimal singularities} if it is less singular than any other $\theta$-psh function.  Such $\theta$-psh functions with minimal singularities always exist, one can consider for example
\begin{equation*}
V_\theta:=\sup\left\{ \varphi\,\,\theta\text{-psh}, \varphi\le 0\text{ on } X \right \}. 
\end{equation*}
Trivially, a $\theta$-psh function with minimal singularities has locally bounded potential in $\Amp(\{\theta\})$.  It follows from Demailly's approximation theorem that $V_{\theta}$ is continuous in the ample locus ${\rm Amp}(\theta)$. 

More generally, if $f$ is a function on $X$, we define the Monge-Amp\`ere envelope of $f$ in the class $\psh(X,\theta)$ by
\[
P_{\theta}(f) := \left(\sup \{u\in \psh(X,\theta) \setdef u\leq f\}\right)^*,
\]
with the convention that $\sup\emptyset =-\infty$. Observe that $P_{\theta}(f)$ is a $\theta$-psh function on $X$ if and only if there exists some $u\in \psh(X,\theta)$ lying below $f$. Note also that $V_{\theta}=P_{\theta}(0)$.

Given $T_1:=\theta_1+dd^c\varphi_1,..., $ $ T_p:=\theta_p+dd^c \varphi_p$   positive $(1,1)$-currents, where $\theta_j$ are closed smooth $(1,1)$-forms, following the construction of Bedford-Taylor \cite{BT87} in the local setting, it has been shown in \cite{BEGZ10} that the sequence of currents 
\[
{\bf 1}_{\bigcap_j\{\varphi_j>V_{\theta_j}-k\}}(\theta_1+dd^c \max(\varphi_1, V_{\theta_1}-k))\wedge...\wedge (\theta_p+dd^c\max(\varphi_p, V_{\theta_p}-k))
\] 
is non-decreasing in $k$ and converges weakly to the so called \emph{non-pluripolar product} 
\[
\langle T_1\wedge\ldots\wedge T_p\rangle= \langle \theta_{\varphi_1 } \wedge\ldots\wedge\theta_{\varphi_1 }\rangle .
\]
The resulting positive $(p,p)$-current does not charge pluripolar sets and it is \emph{closed}. The particular case when $T_1=\cdots =T_p$ will be important for us in the sequel. For a $\theta$-psh function $\varphi$, the \emph{non-pluripolar complex Monge-Amp{\`e}re measure} of $\varphi$ is
$$
\theta_\varphi^n:=\langle(\theta+dd^c\varphi)^n\rangle.
$$
The volume of a big class $\{ \theta\}$  is defined by 
\[
{\rm Vol}(\{\theta\}):= \int_{{\rm Amp}(\{\theta\})} \theta_{V_\theta}^n.
\]

Alternatively, by \cite[Theorem 1.16]{BEGZ10}, in the above expression one can replace $V_\theta$ with any $\theta$-psh function with minimal singularities. A $\theta$-psh function $\varphi$ is said to have \emph{full Monge-Amp\`ere mass} if
\[
\int_X \theta_\varphi^n=\vol(\{\theta\}),
\]
and we then write $\varphi\in \mathcal{E}(X,\theta)$. Let us stress that since the non-pluripolar product does not charge pluripolar sets, for a general $\theta$-psh function $\varphi$ we only have $\vol(\{\theta\})\geq \int_X  \theta_\varphi^n$.   

By a \emph{weight function}, we mean a smooth increasing function $\chi:\mathbb{R}\to\mathbb{R}$ such that $\chi(0)=0$ and $\chi(-\infty)=-\infty$. 
We say that $\varphi \in {\rm PSH}(X,\theta)$ has finite $\chi$-energy if 
$$
E_\chi(\varphi) :=\int_X(-\chi)(\varphi-V_{\theta}) \theta_\varphi^n <+\infty.
$$
We denote by $\mathcal{E}_{\chi}(X,\theta)$ the set of full mass $\theta$-psh potentials having finite $\chi$-energy. If $\mathcal W^-$ denotes the set of weights $\chi$ that are convex on $\mathbb{R}^-$, then by \cite[Proposition 2.11]{BEGZ10} we have
\begin{equation}\label{eq: E_union}
\mathcal E(X,\theta) := \bigcup_{\chi \in \mathcal W^-} \mathcal E_\chi(X,\theta).
\end{equation}
In the special case when $\chi=Id$ we simply denote  the space $\mathcal{E}_{\chi}(X,\theta)$ by $\mathcal{E}^1(X,\theta)$. 
When $\varphi \in \textup{PSH}(X,\theta)$ has the same singularity type as $V_\theta$, the Monge-Amp\`ere energy (in the class $\{\theta\}$) is defined by the formula
\[
{\rm I}(\varphi):= \frac{1}{(n+1){\rm Vol}(\theta)}\sum_{k=0}^n \int_X (\varphi-V_{\theta})\langle \theta_{\varphi}^k \wedge \theta_{V_{\theta}}^{n-k}\rangle.
\]
For general $\varphi \in \textup{PSH}(X,\theta)$, using the monotonicity property of $\AM$, we have the following definition:
$$\AM(\varphi) := \lim_{k \to \infty}\AM(\max(\varphi,V_\theta - k)).$$
Though the above limit exists, it is possible that $\AM(\varphi)=-\infty$. It was proved in \cite{BEGZ10} that $\varphi\in \mathcal{E}^1(X,\theta)$ if and only if ${\rm I}(\varphi)$ is finite.  Moreover, ${\rm I}$ is continuous under monotone and uniform convergence  \cite[Proposition 2.10,Theorem 2.17]{BEGZ10}.

\subsection{Degenerate complex Monge-Amp\`ere equation}
We summarize recent results on the resolution of degenerate complex Monge-Amp\`ere equations in big cohomology classes. The main source are \cite{BEGZ10} and \cite{BBGZ13}. Let $\mu$ be a non-pluripolar  measure on $X$, i.e. a positive Borel measure that puts no mass on pluripolar sets. We want to solve the complex  Monge-Am\`ere equation 
\begin{equation}
\label{eq: BBGZ 1}
\theta _{\varphi}^n =e^{\beta \varphi} \mu, \ \ \varphi \in \mathcal E^1(X,\omega), 
\end{equation}
where $\beta>0$ is a constant. As the treatment is the same for all $\beta$, we assume that $\beta=1$.  If $\mu$ is a smooth volume form (or it has bounded density with respect to $\omega^n$) then one can use a fixed point argument \cite[Section 6.1]{BEGZ10} to solve the equation.  For the general case we use the variational method of \cite{BBGZ13}  to show existence of a solution $\varphi\in \mathcal{E}^1(X,\theta)$.  The proof that we give below is extracted from \cite{BBGZ13},  except  for the argument of Theorem \ref{thm: derivative}, which is inspired from \cite{LN15}.  The main point is to make it clear that the result is independent of \cite{BD12}. 

\subsubsection{The variational method.}
Let $\mu$  be a non pluripolar positive measure on $X$. For convenience we normalize  $\theta$ so that  its volume is $1$.

Consider the following functional 
\[
F(\varphi) := {\rm I}(\varphi) -L_{\mu}(\varphi), \ \ \varphi \in \textup{PSH}(X,\omega),
\]
where  $L_{\mu}(\varphi):=\int_X e^{\varphi} d\mu$. It follows from \cite{BEGZ10} that  ${\rm I}$ is upper semicontinuous with respect to $L^1$-convergence.  Assume that  $\varphi_j$ is a sequence of $\theta$-psh functions converging in $L^1(X,\omega^n)$ to $\varphi\in {\rm PSH}(X,\theta)$. Then by Hartogs lemma $\sup_X \varphi_j$ is uniformly bounded, hence the sequence $e^{\varphi_j}$ stays in ${\rm PSH}(X,A\omega)$ for some positive constant $A$ and  $e^{\varphi_j}$ converges to $e^{\varphi}$ in $L^1(X,\omega^n)$. Since $\mu$ is non pluripolar it thus follows from an argument  due to Cegrell \cite[Lemma 5.2]{Ce98} (see \cite[Theorem 3.10]{BBGZ13} or \cite[Lemma 11.5]{GZ17}  for a proof in the compact case) that $L_{\mu}(\varphi_j)\to L_{\mu}(\varphi)$. This means $L_{\mu}$  is continuous with respect to the $L^1$-topology, hence $F$ is upper semicontinuous on $\mathcal{E}^1(X,\theta)$. 
\begin{prop}
\label{prop: existence of maximizer}
There exists $\varphi\in \mathcal{E}^1(X,\theta)$ such that $F(\varphi)=\sup_{\psi\in \mathcal{E}^1(X,\theta)}F(\psi)$. 
\end{prop}
\begin{proof}
Let $(\varphi_j)$ be a sequence in $\mathcal{E}^1(X,\theta)$ such that $\lim_{j}F(\varphi_j)=\sup_{\mathcal{E}^1(X,\theta)}F>-\infty$. We claim that $\sup_X \varphi_j$ is uniformly bounded from above. Indeed, assume that it were not the case. Then by relabeling the sequence we can assume that $\sup_X \varphi_j$ increases to $+\infty$. By compactness property \cite[Proposition 2.7]{GZ05} it follows that the sequence $\psi_j:=\varphi_j-\sup_X \varphi_j$ converges in $L^1(X,\omega^n)$ to some $\psi\in \psh(X,\theta)$ such that $\sup_X \psi=0$.  In particular $\int_X e^{\psi} d\mu >0$. It thus follows that 
\[
\int_X e^{\varphi_j}d\mu = e^{\sup_X \varphi_j}  \int_X e^{\psi_j} d\mu  \geq c e^{\sup_X \varphi_j}
\]
for some positive constant $c$. Since ${\rm I}(\varphi_j)\leq \sup_X\varphi_j$, the above inequality gives that $F(\varphi_j)$ converges to $-\infty$, a contradiction. Thus $\sup_X \varphi_j$ is bounded from above as claimed. Since $F(\varphi_j)\leq {\rm I}(\varphi_j)\leq \sup_X \varphi_j$ it follows that ${\rm I}(\varphi_j)$ and hence $\sup_X \varphi_j$ is also bounded from below. It follows again from \cite[Proposition 2.7]{GZ05} that a subsequence of $\varphi_j$ (still denoted by $\varphi_j$) converges in $L^1(X,\omega^n)$ to some $\varphi\in \psh(X,\theta)$. Since ${\rm I}$ is upper semicontinuous  it follows that $\varphi\in \Ec^1(X,\theta)$. Moreover, by continuity of $L_{\mu}$ it follows that $F(\varphi) \geq \sup_{\Ec^1(X,\theta)} F$ completing the proof. 
\end{proof}

Next we prove that the maximizer obtained above is actually the solution to the complex Monge-Amp\`ere equation \eqref{eq: BBGZ 1}. The proof relies on a differentiability property of the Monge-Amp\`ere energy functional: 

\begin{theorem}
\label{thm: derivative}
Fix $\varphi\in \mathcal{E}^1(X,\theta)$ and let $\chi$ be a continuous real valued function $X$. Set $\varphi_t=P_{\theta}(\varphi+t\chi), \ t\in \mathbb{R}$. Then $t\mapsto I(\varphi_t)$ is differentiable and its derivative is given by 
\[
\frac{d}{dt}{\rm I}(\varphi_t) = \int_X \chi \theta_{\varphi_t}^n, \ \forall t\in \mathbb{R}. 
\]\end{theorem}

Note that $\varphi_t\geq \varphi - |t|\sup_X |\chi|$, hence $\varphi_t \in \mathcal{E}^1(X,\theta)$ for all $t\in \mathbb{R}$.
This result was first proved in  \cite[Lemma 4.2]{BBGZ13} using \cite{BB08}. A simplification of the original argument has been given in \cite{LN15}, and we follow this approach here.

\begin{proof}
Let $u$ be a continuous function on $X$ and set $u_t:=P_{\theta}(u+t\chi)$, $t\in \mathbb{R}$. Then for each $t\in \mathbb{R}$, $u_t$ is a $\theta$-plurisubharmonic function with minimal singularities. We claim that 
\[
\frac{d}{dt} {\rm I}(u_t) = \int_X \chi \theta_{u_t}^n, \ \forall t\in \mathbb{R}. 
\]
It suffices to prove the claim for $t = 0$. We only prove the equality for the right-derivative since the same argument can be applied to deal with the left derivative. We fix $t > 0$. It follows from the concavity of ${\rm I}$ \cite{BEGZ10} that 
\[
\int_X (u_{t}-u_0) \theta_{u_t}^n \leq {\rm  I}(u_{t})-{\rm I}(u_0) \leq \int_X (u_{t}-u_{0}) \theta_{u_0}^n.
\]
On the other hand, since $u+t\chi$ is continuous on $X$, it follows from a balayage argument (see \cite{BT82}) that $\theta_{u_t}^n$ is supported on the contact set $\{u_t=u+t\chi\}$. It thus follows that 
\[
\int_X (u_{t}-u_0) \theta_{u_t}^n = \int_X (u + t\chi-u_0) \theta_{u_t}^n \geq t \int_X \chi \theta_{u_t}^n. 
\]
By a similar argument we also get
\[
\int_X (u_{t}-u_0) \theta_{u_0} ^n = \int_X (u_t -u) \theta_{u_0} ^n \leq t \int_X \chi \theta_{u_0} ^n. 
\]
We note that $u_t$ converges uniformly to $u_0$ as $t\to 0^+$, hence by \cite{BEGZ10} we have that $\theta_{u_t}^n$ converges weakly to $\theta_{u_0}^n$. Since $\chi$ is continuous, we can divide all of the above estimates with $t>0$, and let $t\to 0$  to finish the proof of the claim.  

Now, we come back to the proof of the theorem. We approximate $\varphi$ from above by a sequence $\varphi_j$ of continuous functions on $X$. For each $j$, we set $\varphi_{t,j}:=P_{\theta}(\varphi_j +t\chi)$ and note that $\varphi_{t,j}$ decreases pointwise to $\varphi_t$ as $j\to +\infty$. Since $\chi$ is continuous on $X$ and  $\varphi_{t,j}$ converges uniformly to $\varphi_{s,j}$ as $t\to s$ it follows from continuity of the complex Monge-Amp\`ere operator together with our claim  that the function $t\mapsto {\rm I}(\varphi_{t,j})$ is of class $\mathcal{C}^1$ on $\mathbb{R}$. We thus have that 
\[
{\rm I}(\varphi_{t,j})-{\rm I}(\varphi_{0,j}) =\int_{0}^t \int_X \chi \theta_{\varphi_{s,j}}^n ds. 
\]
Letting $j\to +\infty$, and using the dominated convergence theorem we obtain 
\[
{\rm I}(\varphi_{t})-{\rm I}(\varphi_{0}) =\int_{0}^t \int_X \chi \theta_{\varphi_{s}}^n ds. 
\]
By continuity of the Monge-Amp\`ere operator the function $s\mapsto \int_X \chi \theta_{\varphi_s}^n$ is continuous on $\mathbb{R}$. Therefore, from the above equality we see that $t\mapsto {\rm I}(\varphi_t)$ is differentiable and its derivative is exactly as in the statement of the theorem. 
\end{proof}
We are now ready to solve the equation \eqref{eq: BBGZ 1}.
\begin{theorem}
\label{thm: maximizers are solutions}
Let $\mu$ be a non pluripolar positive measure on $X$. Then there exists a unique $\varphi\in {\mathcal{E}^1}(X,\theta)$  solving \eqref{eq: BBGZ 1}. Moreover, if $\mu=f\omega^n$ for some bounded nonnegative function $f$ then $\varphi$ has minimal singularities.
\end{theorem}
\begin{proof}
As before we can assume that $\beta=1$. It follows from Proposition \ref{prop: existence of maximizer} that there exists $\varphi\in \mathcal{E}^1(X,\theta)$ such that $F(\varphi)=\sup_{\mathcal{E}^1(X,\theta)} F$. Fix a continuous function $\chi$ and set
\[
g(t):=I(\varphi_t) -\int_X e^{\varphi +t\chi} d\mu, \ t\in \mathbb{R},
\]
where $\varphi_t:=P_{\theta}(\varphi+t\chi)\in \mathcal{E}^1(X,\theta)$. Since $\varphi_t\leq \varphi+t\chi$ and since $\varphi$ maximizes $F$ on $\mathcal{E}^1(X,\theta)$ we have 
\[
g(0) =F(\varphi) \geq F(\varphi_t) \geq g(t), \ \forall t\in \mathbb{R}.  
\]
Thus $g$ attains its maximum at $0$. It follows from Theorem \ref{thm: derivative} and the dominated convergence theorem that $g$ is differentiable at $0$, hence $g'(0)=0$ which means 
\[
\int_X \chi \theta_{\varphi}^n =\int_X \chi e^{\varphi} d\mu.  
\]
Since $\chi$ was chosen arbitrarily it follows that the equation \eqref{eq: BBGZ 1} is satisfied in the weak sense of measure. 

The uniqueness follows from Lemma \ref{lem: domination principle} below. Finally if $\mu$ has bounded density then the right-hand side in \eqref{eq: BBGZ 1} has bounded density (because $\varphi$ is bounded from above)  hence it follows from \cite[Thereom 4.1]{BEGZ10} that $\varphi$ has minimal singularities. 
\end{proof}

\subsubsection{The domination principle.}
The following domination principle was proved in \cite[Corollary 2.5]{BEGZ10} for two $\theta$-psh functions $\varphi,\psi$ with $\varphi$ having minimal singularities. The argument of  Dinew \cite{BL12}  gives a generalization of this result to the case when $\varphi\in \Ec(X,\theta)$  does not necessarily have minimal singularities. 

\begin{prop}
        \label{prop: domination principle}
        Let $\varphi,\psi$ be $\theta$-psh functions such that $\varphi\in \Ec(X,\theta)$. If $\psi\leq \varphi$ almost everywhere with respect to $\theta_{\varphi}^n$ then $\psi\leq \varphi$ everywhere. 
\end{prop}
\begin{proof}
Fix $t>0$. As $\{\varphi > \psi -t \}$ is plurifine open, it follows from locality of the non-pluripolar product with respect to the plurifine topology \cite[Proposition 1.4]{BEGZ10} that
\[
\theta_{\max(\varphi,\psi-t)}^n \geq {\bf 1}_{\{\varphi>\psi-t\}}\theta_{\max(\varphi,\psi-t)}^n = {\bf 1}_{\{\varphi>\psi-t\}} \theta_{\varphi}^n=\theta_{\varphi}^n,
\]
where in the last identity we used the assumption that $\theta_\varphi^n(\{\varphi< \psi\})=0$.  
As $\max(\varphi,\psi-t)$ also has full mass, the above inequality becomes equality. This together with the uniqueness theorem, proved in \cite[Theorem A]{BEGZ10} which is an adaptation of the original proof in the K\"ahler case due to Dinew \cite{Dw09}, gives that $\max(\varphi,\psi-t)=\varphi+C$, for some constant $C$ which can be easily seen to be zero.  Letting $t\to 0^+$ we obtain the desired result.
\end{proof}

\begin{lemma}
        \label{lem: domination principle}
        Fix $\beta>0$ and $\mu$ a non-pluripolar positive measure on $X$. Assume that $v\in \mathcal{E}(X,\theta)$ and $u\in {\rm PSH}(X,\theta)$ satisfy 
        \[
        \theta_{u}^n \geq e^{\beta u} e^{-\phi}\mu \ ; \ \theta_{v}^n\leq e^{\beta v} e^{-\phi}\mu,
        \]
        where $\phi$ is some Borel measurable function on $X$. Then $u\leq v$ on $X$. 
\end{lemma}
\begin{proof}
        It follows from the comparison principle (see \cite[Corollary 2.3]{BEGZ10}) that 
\[
        \int_{\{v<u\}} e^{\beta u -\phi} d\mu \leq \int_{\{v<u\}} \theta_u^n \leq \int_{\{v<u\}} \theta_v^n \leq \int_{\{v<u\}} e^{\beta v-\phi} d\mu \leq \int_{\{v<u\}}e^{\beta u-\phi} d\mu. 
\]
Thus, the inequalities above become equalities and we have in particular that 
\[
\int_{\{v<u\}} (e^{\beta u}-e^{\beta v})e^{-\phi}d\mu =0.
\]
Since $\phi$ is bounded from above on $X$ it follows that $\int_{\{v<u\}} (e^{\beta u}-e^{\beta v})d\mu=0$. We then deduce that $\mu(\{-\infty<v<u\})=0$. Since $\mu$ is non-pluripolar it follows that  $\mu(\{v<u\})=0$. Consequently  $\theta_v^n(\{v<u\})=0$, and the domination principle (Proposition \ref{prop: domination principle}) gives that $u\leq v$ on $X$. 
\end{proof}

\subsection{Regularity of quasi-psh envelopes}

By a deep result of Berman and Demailly \cite{BD12}, the envelope $V_{\theta}$ has locally bounded Laplacian in the ample locus ${\rm Amp}(\{\theta\})$ and its complex Monge-Amp\`ere measure satisfies 
\[
\theta_{V_{\theta}}^n = {\bf 1}_{\{V_{\theta}=0\}} \theta^n. 
\]
In the case when $\{\theta\}$ is integral, this result has been obtained by Berman \cite{Ber08} using different methods.
In this subsection we establish a weaker version of this result using viscosity methods combined with the convergence method of Berman \cite{Ber13}. We warmly thank Robert Berman for stimulating discussions and for sharing his ideas on the argument we provide below. 

\begin{theorem}Suppose $(X,\omega)$ is K\"ahler. \label{thm: weak BD}
        Let $\theta$ be a smooth $(1,1)$-form such that $\{\theta \}$ is big. Then the envelope $V_{\theta}$ satisfies 
\begin{equation}\label{eq: BD_weak_est}
        \theta_{V_{\theta}}^n \leq {\bf 1}_{\{V_{\theta}=0\}} \theta^n.
\end{equation}
\end{theorem}


Before giving the details of the proof, we briefly explain the main points. We introduce $\theta^n_+ = \max\big(\frac{\theta^n}{\omega^n},0\big) \omega^n$, and for each $\beta>0$ Theorem \ref{thm: maximizers are solutions} provides a unique $\varphi_{\beta}\in \psh(X,\theta)$ with minimal singularities solving
\[
\theta_{\varphi_{\beta}}^n =e^{\beta\varphi_{\beta}} \theta_+^n.
\] 

As we show below, when $\beta\to +\infty$ the solutions $\varphi_{\beta}$ converge increasingly to the envelope $V_{\theta}$. As simple balayage argument provides that $\theta_{V_\theta}^n$ is supported on $\{V_\theta =0 \}$, to prove Theorem \ref{thm: weak BD} it suffices to show that $\varphi_{\beta}\leq 0$ for all $\beta>0$. 

We now describe a naive argument confirming this last fact that will be used as a guide in the proof below. Assume that the maximum of $\varphi_{\beta}$ on $X$ is attained at $x_0$. If we could apply the maximum principle to the function $\varphi_{\beta}$, then $dd^c \varphi_{\beta}$ would be negative semidefinite at $x_0$. Thus from the equation above we could deduce that $\varphi_{\beta}(x_0)\leq 0$. Of course we can not apply the maximum principle in this naive way because the solution $\varphi_{\beta}$ may not be $\Cc^2$. Instead we will use more rigid viscosity methods, originally developed by P.-L. Lions and others in the eighties (see e.g. \cite{IL90}) and adapted to complex Monge-Amp\`ere equations recently  (see \cite{EGZ11, EGZ16}, \cite{HL13}, \cite{Lu13}, \cite{Wang12},\cite{Zer13}, to only mention a few works in a very fast expanding literature). 
\begin{proof}[Proof of Theorem \ref{thm: weak BD}]As $\{\theta\}$ is big, we can assume that there exists $\psi \in {\rm PSH}(X,\theta)$ with analytic singularities  such that $\theta +dd^c \psi \geq \omega$. For each $\varepsilon>0, \beta>1$ it follows from Theorem \ref{thm: maximizers are solutions}  that there exists a unique $\varphi_{\beta,\varepsilon}\in {\rm PSH}(X,\theta)$  with minimal singularities such that 
 \begin{equation}
        \label{eq: Berman convergence 1}
        \theta_{\varphi_{\beta,\varepsilon}}^n = e^{\beta \varphi_{\beta,\varepsilon}} \left[(1+\varepsilon)\theta^n_+ + \varepsilon \omega^n \right].
 \end{equation}
In particular $\varphi_{\beta,\varepsilon}$ is locally bounded in the ample locus $\Omega:={\rm Amp}(\{\theta\})$. 

We claim that $\varphi_{\beta,\vep}\leq 0$ for all $\beta>1, \varepsilon>0$. Due to the lack of regularity for $\varphi_{\beta,\varepsilon}$ we will need some techniques from viscosity theory that we detail now. To prove the claim, we observe that it suffices to prove that 
\begin{equation}
        \label{eq: Berman convergence 2}
        v_{\lambda}:=(1-\lambda)\varphi_{\beta,\varepsilon} +\lambda \psi \leq 0,
\end{equation}
for all $\lambda\in (0,1)$ such that $(1-\lambda)^n(1+\varepsilon)>1$. Indeed, if the above estimate holds then we let $\lambda\to 0^+$ and obtain $\varphi_{\beta,\varepsilon}\leq 0$ in the ample locus $\Omega$. By basic properties of plurisubharmonic functions, this inequality extends to all of $X$ (the complement of $\Omega$ has Lebesgue measure zero). 

Now we argue \eqref{eq: Berman convergence 2}. Fixing $\beta,\varepsilon>0$, by possible adding a constant, we can assume that $\psi \leq \varphi_{\beta,\varepsilon}$. By contradiction, we assume that there exists $x_0\in X$ such that $c:=v_{\lambda}(x_0) = \sup_X v_\lambda >0$. Since $\psi|_{X \setminus \Omega}=-\infty$, it follows that $x_0\in \Omega$. In particular, $\psi$ is smooth in a small neighborhood  of $x_0$. 

Next, fix a ball $B(x_0,r)\subset \Omega$ in a holomorphic coordinate chart around $x_0$. By possibly shrinking $B(x_0,r)$, we can assume that the local potential $g$ of $\theta$ (i.e. $dd^c g=\theta$ in $B(x_0,r)$) satisfies $-c\leq g\leq 0$ in $B(x_0,r)$. In this ball the function $u:=v_{\lambda} + g$ is plurisubharmonic, bounded, and  we can write the following sequence of estimates:
\begin{eqnarray*}
(dd^c u)^n& = & (\theta +dd^c v_{\lambda})^n = \left[ (1-\lambda) \theta_{\varphi_{\beta,\varepsilon}} + \lambda \theta_{\psi}\right]^n  \\
&\geq & (1-\lambda)^n \theta_{\varphi_{\beta,\varepsilon}}^n = (1-\lambda)^n e^{\beta \varphi_{\beta,\varepsilon}} \left[(1+\varepsilon)\theta_{+}^n+\varepsilon \omega^n\right] \\
&\geq &e^{\beta u} \left[\theta^n_+ + \varepsilon(1-\lambda)^n \omega^n \right], 
\end{eqnarray*}
where in the last line we used that $\varphi_{\beta,\varepsilon} \geq v_{\lambda} =u-g\geq u$ in $B(x_0,r)$ and $(1-\lambda)^n(1+\varepsilon) >1$. 

On the other hand the function $u-g-c$, defined in $B(x_0,r)$, attains a maximum at $x_0$ (equal to zero). 
It follows from Lemma \ref{lem: viscosity vs pluripotential local} below that at $x_0$ we have 
\[
\theta^n = (dd^c g)^n \geq e^{\beta(g+c)} \left[\theta_+^n + \varepsilon(1-\lambda)^n \omega^n\right]. 
\]
But this is a contradiction since $g+c\geq 0$  and $\omega^n>0$ in $B(x_0,r)$. Thus \eqref{eq: Berman convergence 2} is proved together with the claim. 

By the comparison principle (Lemma \ref{lem: domination principle}) the solutions $\varphi_{\beta,\varepsilon} \leq 0$ are increasing as $\varepsilon \searrow 0$. Indeed, if $s<\varepsilon$ then 
\[\theta_{\varphi_{\beta,\varepsilon}}^n \geq e^{\beta \varphi_{\beta,\varepsilon}} \left[(1+s)\theta_+^n +s\omega^n\right],\]
hence we can use Lemma \ref{lem: domination principle} with $\mu:=(1+s)\theta_+^n +s\omega^n$ and $\phi=0$ to conclude that $\varphi_{\beta,\varepsilon}\leq \varphi_{\beta,s}$.  

Since $\varphi_{\beta,\varepsilon}\leq 0$ for all $\beta,\varepsilon>0$ it follows that $\varphi_{\beta,\varepsilon}$ increases almost everywhere to some $0\geq \varphi_{\beta}\in {\rm PSH}(X,\theta)$ with minimal singularities such that
\[
\theta _{\varphi_{\beta}}^n = e^{\beta \varphi_{\beta}} \theta_+^n. 
\]
Using Theorem \ref{thm: maximizers are solutions} again, there exists a unique $\phi \in {\rm PSH}(X,\theta)$ with minimal singularities such that 
\[
\theta_{\phi}^n =e^{\phi}\theta_+^n.
\]
By the comparison principle for $\beta > 1$ we have that 
\[
\varphi_{\beta} \geq u_\beta :=\left(1-\frac{1}{\beta}\right) V_{\theta} + \frac{1}{\beta} \phi -\frac{n\log \beta}{\beta}. 
\]
Indeed, $u_{\beta}\in {\rm PSH}(X,\theta)$ has minimal singularities and
\[
\theta_{u_{\beta}}^n \geq \frac{1}{\beta^n} \theta_{\phi}^n = \frac{1}{\beta^n} e^{\phi}\theta_{+}^n =e^{\phi - n \log \beta}\theta_{+}^n\geq e^{\beta u_{\beta}} \theta_+^n.  
\]
It thus follows from Lemma \ref{lem: domination principle} that $u_{\beta} \leq \varphi_{\beta}$ as claimed. 

Since $\varphi_\beta \leq 0$, it again follows from the comparison principle that $\varphi_{\beta}$ is increasing in $\beta$. Hence $\varphi_{\beta}\nearrow V_{\theta}$ and by continuity of the Monge-Amp\`ere operator we have 
\[
\theta _{V_{\theta}}^n = \lim_{\beta\to +\infty} \theta_{\varphi_{\beta}}^n \leq \theta_{+}^n. 
\]
Finally, by a standard balayage argument, $\theta_{V_{\theta}}^n$ is supported on the contact set $\{V_{\theta}=0\}$. Alternatively, this can be seen as follows. For each $\delta>0$, since $U:=\{V_{\theta}<-\delta\}$ is an open set, we have 
\[
\int_{U} \theta_{V_{\theta}}^n \leq \liminf_{\beta \to+\infty} \int_U \theta_{\varphi_{\beta}}^n \leq \liminf_{\beta\to +\infty} \int_U e^{-\beta \delta} \theta_+^n =0. 
\]
Putting everything together we obtain that $\theta_{V_\theta}^n \leq \mathbf{1}_{\{V_\theta =0\}}\theta^n_+$. As the very last step,  we see that $\mathbf{1}_{\{V_\theta =0\}}\theta^n_+ = \mathbf{1}_{\{V_\theta =0\}}\theta^n$. Indeed, this follows from an application of  next Lemma for any $x_0 \in \{V_\theta =0\}$, $u := V_\theta + g$, and $q:=g$, where $g$ is a local potential of $\theta$ near $x_0$.
\end{proof}

The following lemma is classical, but as a courtesy to the reader, we provide a proof extracted from \cite{HL09}: 

\begin{lemma}
\label{lem: viscosity subsolution}
Let $U$ be an open subset in $\mathbb{C}^n$. Assume that $u\in {\rm PSH}(U)$ and $q\in \mathcal{C}^2(U)$. If $u-q$ attains its maximum at $x_0\in U$ then $dd^c q(x_0) \geq 0$.
\end{lemma}
 
\begin{proof}
The result follows from the maximum principle, if $u$ is also of class $\Cc^2$ near $x_0$. Approximate $u$ by a decreasing sequence of smooth psh functions $u_j$ in a smaller domain $U_1 \Subset U$. Let $B(x_0,r) \Subset \Omega_1$ for some $r>0$ and fix $\vep>0$. Let $x_j\in \bar{B}(x_0,r)$ be  a point where the function
$$
\varphi_j (x) := u_j(x) - q(x) -\vep|x-x_0|^2 , \ x\in \bar{B}(x_0,r)
$$ 
attains its maximum. We claim that $x_j \to x_0$. Assume, after extracting a subsequence if necessary, that $x_j\to \bar{x}$. Since $\varphi_j$ attains its maximum at $x_j$ it follows that
$$
u_j(x_0) -q(x_0) \leq u_j(x_j) -q(x_j) -\vep|x_j-x_0|^2.
$$
For each $k\in \mathbb{N}$ fixed, since the functions $u_k$ is smooth in $U_1$ it follows that $\limsup_{j} u_j(x_j) \leq \lim_{j}u_k(x_j) = u_k(\bar{x})$. Letting $k\to+\infty$ we obtain $\limsup_j u_j(x_j)\leq u(\bar{x})$. As $j\to +\infty$ one gets
$$
u(x_0) -q(x_0) \leq u(\bar{x}) -q(\bar{x}) -\vep|\bar{x}-x_0|^2 \leq u(x_0) -q(x_0) -\vep |\bar{x}-x_0|^2,
$$
hence $\bar{x}=x_0$ yielding the claim. For $j$ big enough, $x_j$ belongs to the interior of the ball. This allows us to apply the maximum principle which gives $dd^c q \geq -\vep  dd^c |x|^2$ at $x_j$. Letting $j\to +\infty$ and then $\vep \to 0$ we see that $dd^c q\geq 0$ at $x_0$.
\end{proof}

\begin{lemma}
        \label{lem: viscosity vs pluripotential local}
        Let $U$ be a bounded domain in $\CC^n$ and fix $\beta>0$. Assume that $u \in \textup{PSH}(U) \cap L^\infty$ such that
        \[
        (dd^c u)^n \geq e^{\beta u} fdV, 
        \]
        where $f>0$ is a continuous function in $U$ and $dV$ is a volume form. If $q$ is a $\Cc^2$ function in a neighborhood of $x_0$ in $U$ such that $u \leq q$ and $u(x_0)=q(x_0)$, then $(dd^c q)(x_0)$ is positive semidefinite and 
        \[
        (dd^c q)^n \geq e^{\beta q} fdV \ \textup{ at } \ x_0. 
        \] 
\end{lemma}
This lemma essentially says that a plurisubharmonic subsolution to the equation $(dd^c v)^n=e^{\beta v} f dV$ is also a viscosity subsolution. The result is well-known for experts in the viscosity theory for complex Monge-Amp\`ere equations (see \cite{EGZ11}), but for the convenience of the reader we give a detailed proof below.

\begin{proof} We can assume without loss of generality that $\beta=1$.  Let us suppose that our conclusion fails, i.e., there exists $x_0\in B:=B(x_0,r)\Subset U$ and $q\in \Cc^2(B)$ such that $u-q$ attains a maximum in $B$ at $x_0$ but $(dd^c q)^n<e^{q}fdV$ at $x_0$.  It follows from Lemma \ref{lem: viscosity subsolution} that $dd^c q$ is positive semidefinite at $x_0$. Denoting $q_{\vep}(x):=q(x) + \vep\|x-x_0\|^2-\vep s^2$, we claim that there exist positive constants $s<<r$ and $\vep>0$ such that 
        \[
        (dd^c q_{\vep})^n \leq e^{q_{\vep}} fdV \ \textup{ on } \ B(x_0,s).
        \]
Indeed, this follows from the fact that $q \in \Cc^2(B)$ and the fact that $(dd^c q)^n<e^{q}fdV$ near $x_0$.       

        Since $u\leq q=q_{\vep}$ on $\partial B(x_0,s)$ it follows from the comparison principle \cite[Theorem 4.1]{BT82} that 
        \[
        \int_{U_{\vep}} e^u fdV \leq \int_{U_{\vep}}(dd^c u)^n \leq \int_{U_{\vep}} (dd^c q_{\vep})^n \leq \int_{U_{\vep}}e^{q_{\vep}} fdV \leq \int_{U_{\vep}}e^{u} fdV,
        \] 
        where $U_{\vep}:= \{x\in B(x_0,s)\setdef q_{\vep}(x) <u(x)\}$. Consequently, all the inequalities above are in fact equalities. In particular, 
        \[
        \int_{U_{\vep}}(e^{u}-e^{q_{\vep}}) fdV=0,
        \]
        which yields $\int_{U_{\vep}} fdV=0$. Since $f>0$ it follows that $U_{\vep}$ has Lebesgue measure zero, hence $u\leq q_{\vep}$ on $B(x_0,s)$. But this is a contradiction, since $q_{\vep}(x_0)=q(x_0)-\vep s^2<u(x_0)$. 
\end{proof}

\begin{remark}\label{rem: BBGZ independent of BD12}
In the general case of big classes, the proof of the main result of the paper \cite{BBGZ13} uses the main regularity result of \cite{BD12} in \cite[Lemma 2.9, Lemma 3.2]{BBGZ13}. But what was actually used in these lemmas is the fact that the complex Monge--Amp\`ere measure of $V_{\theta}$ has bounded density with respect to Lebesgue measure and the latter follows from Theorem \ref{thm: weak BD}, that we proved above, without using \cite{BD12}.  
\textbf{}\end{remark}

\subsection{Comparison of capacities}
Given a big class $\{\theta\}$, recall that the $\theta$-capacity of a set $E \subset X$ is defined as follows (see \cite[Section 4.1]{BEGZ10} for further details):
\[
\capa_{\theta}(E) : = \sup \left\{\int_E \theta_u^n\setdef u\in \psh(X,\theta) \quad V_{\theta}-1\leq u\leq V_{\theta}\right\}.
\] The global extremal $\theta$-psh function of $E$ is defined as the usc regularization of 
\[
V_{\theta,E}:=\sup \left\{ u\in \psh(X,\theta) \setdef u\leq 0 \ {\rm on}\ E \right\}.
\]

By the definition of $V_{\theta,E}$ it follows that $V_{\theta,E}^*=V_{\theta,F}^*$ if $E=F\cup P$ for some pluripolar set $P$.  To see this it suffices to observe that any pluripolar set is contained in the $-\infty$ locus of  some $\varphi \in \psh(X,\theta)$. The latter standard fact can be quickly explained as follows. If $E$ is pluripolar then by \cite[Theorem 7.2]{GZ05} there exists $u\in \psh(X,\omega)$ such that $u(x)=-\infty$ for all $x\in E$. As $\{\theta\}$ is big, there exists $\psi \in \psh(X,\theta)$ such that $\theta+dd^c \psi \geq \vep \omega$ for some small constant $\vep>0$. It is clear that the function $\varphi:=\psi +\vep u$ belongs  to $\psh(X,\theta)$ and takes value $-\infty$ on $E$ as we claimed. 

Denote by $T_\theta(E) := \exp(-M_{\theta,E}):=\exp(- \sup_X V^\star_{\theta,E})$ the \emph{Alexander--Taylor capacity} of $E$. We recall the following useful relation between the $\theta$-capacity and the Alexander--Taylor capacity (see \cite[Proposition 5.4]{DN15} and \cite[Lemma 4.2]{BEGZ10}): 

\begin{prop}\label{comparison}

There exists $A>0$ such that for all Borel subsets $E\subset X$,
$$\exp\left[-\frac{A}{\capa_{\theta}(E)}        \right] \leq T_\theta(E)\leq e \cdot \exp \left[-\left(\frac{\vol(\{\theta\})} {\capa_{\theta}(E)}\right)^{\frac{1}{n}}  \right] .$$ 
\end{prop}

The second inequality was proved in \cite{BEGZ10} while the first one was proved in \cite[Proposition 5.4]{DN15} using the main result of \cite{BD12}. But, what was actually used in the proof of  \cite[Proposition 5.4]{DN15} is the inequality in Theorem \ref{thm: weak BD}. Thus Proposition \ref{comparison} is independent of \cite{BD12}.

Using Proposition \ref{comparison} and a sharp analysis in the ample locus we get the following comparison between two $\theta$-capacities. Let us emphasize that this result significantly extends \cite[Theorem 5.6]{DN15} where an extra assumption  is required (see \cite[Definition 4.2]{DN15}).
\begin{theorem}\label{thm: comparison of capacity} Suppose $\{\theta_1\}$ and $\{\theta_2\}$ are big. Then there exists $C=C(\theta_1,\theta_2)>0$ such that
        \[
        C^{-1}\capa_{\theta_1}^{n}(E)  \leq \capa_{\theta_2}(E) \leq C\capa_{\theta_1}^{1/n}(E), 
        \] 
        for all Borel subsets $E\subset X$.
\end{theorem}
\begin{proof}Since $\{\theta_2\}$ is big we can find $\psi\in \psh(X,\theta_2)$ that is smooth in $\textup{Amp}(\{\theta_2\})$, $\psi$ has analytic singularities such that $\theta_2+dd^c \psi \geq \vep \omega\geq \vep \theta_1$ for some $\vep>0$ (\cite{Bou04}). Normalize $\psi$ by $\sup_X \psi=0$ and denote by $U=\{\psi>-1\}$ which is a nonempty open subset of $X$, hence $M_{\theta_1,U}=\sup_X V_{\theta_1,U} <+\infty$. Now, the function $u=\psi + \vep V_{\theta_1,E}^*$ is $\theta_2$-psh and satisfies $u\leq 0$ on $E$ (modulo a pluripolar set). Thus by definition we have $u\leq V_{\theta_2,E}^*$. It follows that 
        \[
        M_{\theta_2,E} \geq \sup_U u \geq \vep \sup_U V_{\theta_1,E}^* -1.
        \]
        On the other hand $V_{\theta_1,E}^* -\sup_U V_{\theta_1,E}^*$ is $\theta_1$-psh and takes nonpositive values on $U$, hence $V_{\theta_1,E}^* -\sup_U V_{\theta_1,E}^* \leq V_{\theta_1,U}$. This together with the above inequality yields
        \[
        M_{\theta_2,E} \geq \vep (M_{\theta_1,E} -M_{\theta_1,U}) -1,
        \]  
Giving that, for some $C>0$ fixed we have $T_{\theta_2}(E) \leq C T_{\theta_1}(E)^\varepsilon$. An elementary calculation using the double estimate of Proposition \ref{comparison} finishes the proof.     
\end{proof}

From this comparison of capacities and standard arguments in pluripotential theory we  immediately get the following convergence result.
\begin{coro}\label{cor: convergence big}
Assume that $\{ \theta\}$ is big and $\{\varphi^i_j\}_j, i=1,...,n$ are sequences  of $\theta$-psh functions with minimal singularities that converge decreasingly (or uniformly) to $\varphi^i, i=1,...,n$ (also with minimal singularities). If $f_j$ is a sequence of uniformly bounded quasi-continuous functions converging monotonically to $f$ (also quasi-continuous) then 
\[
\int_X f_j \theta_{\varphi^1_j} \wedge ... \wedge \theta_{\varphi^n_j} \to \int_X f \theta_{\varphi^1} \wedge ... \wedge \theta_{\varphi^n}. 
\]
\end{coro}
\begin{proof}Let $\omega$ be a K\"ahler form on $X$. By the definition of quasi-continuity, for any $\varepsilon >0$ one can find an open set $U \subset X$ such that all $f_j$ are continuous in $X \setminus U$ and $\textup{Cap}_\omega(U) \leq \varepsilon$. Using Theorem \ref{thm: comparison of capacity}, a standard argument now gives that 
\begin{equation}\label{eq: capintest}
\int_U \theta_{\varphi^1_j} \wedge ... \wedge \theta_{\varphi^n_j} \leq C \varepsilon^{1/n}, \ \ \ \int_U \theta_{\varphi^1} \wedge ... \wedge \theta_{\varphi^n} \leq C \varepsilon^{1/n}.
\end{equation}
Now we can use Tietze's  theorem to extend each $f_j|_{X \setminus U},f|_{X \setminus U}$ to a continuous functions $\tilde f_j,\tilde f$ on $X$ whose $L^\infty$ norm is controlled. It follows from \cite[Theorem 2.17]{BEGZ10} that for $j_0$ fixed we have 
$$\int_X \tilde f_{j_0} \theta_{\varphi^1_j} \wedge ... \wedge \theta_{\varphi^n_j} \to \int_X \tilde f_{j_0} \theta_{\varphi^1} \wedge ... \wedge \theta_{\varphi^n}, \ \ \int_X \tilde f \theta_{\varphi^1_j} \wedge ... \wedge \theta_{\varphi^n_j} \to \int_X \tilde f \theta_{\varphi^1} \wedge ... \wedge \theta_{\varphi^n}.$$ 
Using \eqref{eq: capintest} and the uniform boundedness of $f_j,\tilde f_j,f,\tilde f$, we can subsequently write that 
$$\int_X f_{j_0} \theta_{\varphi^1_j} \wedge ... \wedge \theta_{\varphi^n_j} \to \int_X f_{j_0} \theta_{\varphi^1} \wedge ... \wedge \theta_{\varphi^n}, \ \ \int_X f \theta_{\varphi^1_j} \wedge ... \wedge \theta_{\varphi^n_j} \to \int_X f \theta_{\varphi^1} \wedge ... \wedge \theta_{\varphi^n}.$$ 
Finally, using the above and the monotonicity of $f_j$ we can write that
$$\lim_{j_0 \to \infty}\int_X f_{j_0} \theta_{\varphi^1} \wedge ... \wedge \theta_{\varphi^n} \geq \lim_{j \to \infty } \int_X f_j \theta_{\varphi^1_j} \wedge ... \wedge \theta_{\varphi^n_j}\geq \int_X f \theta_{\varphi^1} \wedge ... \wedge \theta_{\varphi^n}.$$
After invoking the dominated convergence theorem, the proof is finished.
\end{proof}

\subsection{The operator $P(\varphi,\psi)$}
Consider $X$ a compact K\"ahler manifold and $\{\theta\}$ a big cohomology class. Given an usc function $f$ on $X$, it is natural to ask whether there exists $u\in \psh(X,\theta)$ lying below $f$ . We will pay particular attention to the case when $f=\min(\varphi,\psi)$, where $\varphi,\psi \in \psh(X,\theta)$, in which case more can be said. Indeed, when $\{\theta\}$ is K\"ahler and $\varphi,\psi\in \Ec(X,\theta)$, it was shown in \cite{Dar14a} that $P_\theta(\varphi,\psi) \in \mathcal E(X,\theta)$. The analogue of this result holds in the big case as well: 
\begin{theorem} \label{thm: sub-extension} Let $\chi \in \mathcal W^-$, i.e., $\chi$ is convex increasing with $\chi(0)=0$ and $\chi(-\infty)=-\infty$. If $\varphi,\psi\in \Ec_\chi(X,\theta)$, then $P_{\theta}(\varphi,\psi):=P_{\theta}(\min(\varphi,\psi))\in \Ec_\chi(X,\theta)$. In particular, if $\varphi,\psi\in \Ec(X,\theta)$, then $P_{\theta}(\varphi,\psi)\in \Ec(X,\theta)$.
\end{theorem}
\begin{proof}
Without loss of generality we can assume that $\varphi,\psi\leq 0$. 
Let $\varphi_j:=\max(\varphi,V_{\theta}-j)$, $\psi_j:=\max(\psi,V_{\theta}-j)$ be the canonical approximants. 
 For each $j>0$, it follows from Lemma \ref{lem: solution unbounded measure} below that there exists a unique $u_j\in \psh(X,\theta)$ with minimal singularities such that
        \begin{equation}
                \label{eq: sub-extension 1}
                \theta_{u_j}^n = e^{u_j-\varphi_j} \theta_{\varphi_j}^n + e^{u_j-\psi_j} \theta_{\psi_j}^n.
        \end{equation}
Additionally, it follows from Lemma \ref{lem: domination principle} (see also the proof of the uniqueness part in Lemma \ref{lem: solution unbounded measure}) that $u_j\leq \min(\varphi_j,\psi_j)$. Consequently $u_j\leq P_\theta(\varphi_j,\psi_j)$. 

Next we claim that  
        \[
                \inf_{j}\int_X \chi(u_j-V_{\theta}) \theta_{u_j}^n >-\infty.  
        \]      
To prove the claim, in view of \eqref{eq: sub-extension 1} it suffices to prove that 
                \begin{equation}
                \label{eq: sub-extension 2}
                \inf_{j}\int_X \chi(u_j-V_{\theta}) e^{u_j-\varphi_j}\theta_{\varphi_j}^n >-\infty.  
        \end{equation}  
By convexity of $\chi$ it follows that $\chi(t+s)\geq \chi(t)+\chi(s)$, and also $\chi(t)e^t \geq -C,$  for all $t,s\leq 0$ for some $C>0$. Thus to prove \eqref{eq: sub-extension 2} it suffices to check that 
        \[
                \inf_{j}\int_X \chi(\varphi_j-V_{\theta}) e^{u_j-\varphi_j}\theta_{\varphi_j}^n >-\infty.  
        \]      
        But this holds since $u_j\leq \varphi_j$ and $\varphi\in \Ec_{\chi}(X,\theta)$.  Thus the claim is proved. 
        
        Since $\chi(-\infty)=-\infty$, the claim implies that $\sup_X u_j$ is uniformly bounded. 
It thus follows from \cite[Proposition 2.19]{BEGZ10}  that some subsequence of $u_j$ converges in $L^1(X,\omega^n)$ to some $u\in \Ec_{\chi}(X,\theta)$. Since $u_j\leq P_\theta(\varphi_j,\psi_j)$ it follows that $u\leq P_{\theta}(\min(\varphi,\psi))$, completing the proof.
\end{proof}

\begin{lemma}
        \label{lem: solution unbounded measure}
        Assume that $u,v\in \psh(X,\theta)$ with minimal singularities. Then there exists a unique $\varphi\in \psh(X,\theta)$ with minimal singularities such that 
        \[
        \theta_{\varphi}^n = e^{\varphi-u}\theta_{u}^n + e^{\varphi-v}\theta_v^n. 
        \] 
\end{lemma}
\begin{proof}
        The uniqueness follows from the Lemma \ref{lem: domination principle}. Indeed, assume that $\psi\in \Ec(X,\theta)$ is another solution, i.e. 
        \[
        \theta_{\psi}^n = e^{\psi-u} \theta_{u} ^n + e^{\psi-v} \theta_v^n.
        \]
        Set $\mu=e^u \theta_v^n +e^{v} \theta_u^n$ and $\phi:=u+v$. Then we can write $\theta_{\varphi}^n =e^{\varphi-\phi} \mu$ and $\theta_{\psi}^n=e^{\psi-\phi}\mu$. It thus follows from Lemma \ref{lem: domination principle} that $\varphi=\psi$. 
        
         To prove the existence we approximate $u$ by $u_j:=\max(u,-j)$ and note that these are bounded functions. Observe that, for each $j$,
        \[
        \mu_j:= e^{-u_j} \theta_u^n + e^{-v_j} \theta_v^n
        \]
        is a non-pluripolar positive measure on $X$. For each $j>0$ it follows from Theorem \ref{thm: maximizers are solutions}  that there exists $\varphi_j\in \psh(X,\theta)$ with minimal singularities such that 
        \[
        \theta_{\varphi_j}^n =e^{\varphi_j}\mu_j. 
        \]
        Let $C>0$ be a constant such that $\sup_X |u-v|\leq 2C$. This is possible because $u$ and $v$ have minimal singularities. The function $\phi:=\frac{u+v}{2} -C-n\log 2$ is $\theta$-psh with minimal singularities and it satisfies
        \[
        \theta _{\phi}^n \geq e^{\phi} \mu_j. 
        \]
        It thus follows from Lemma \ref{lem: domination principle} that $\varphi_j\geq \phi$ for all $j$. It also follows from Lemma \ref{lem: domination principle} that $\varphi_j$ is decreasing in $j$, the pointwise limit $\varphi:=\lim_{j\to +\infty} \varphi_j$ has minimal singularities. By continuity of the Monge-Amp\`ere operator (see \cite{BEGZ10}) it follows that  $\varphi$ is the solution we are looking for. 
\end{proof}

As a simple consequence of the above result, we can settle the conjecture of \cite[Remark 2.16]{BEGZ10} about the convexity of the classes $\mathcal E_\chi(X,\theta)$. A similar result in the K\"ahler case was obtained in \cite{Dar14a}.
\begin{coro}\label{cor: BBEGZ_conj} Suppose $\chi \in \mathcal W^- $. Then $\mathcal E_\chi(X,\theta)$ is convex.
\end{coro}
\begin{proof} Given $u_0,u_1 \in \mathcal E_\chi(X,\theta)$ it follows that $P_\theta(u_0,u_1) \leq tu_0 + (1-t)u_1, \ t \in [0,1]$. By the above result $P_\theta(u_0,u_1) \in \mathcal E_\chi(X,\theta)$. Now \cite[Proposition 2.14]{BEGZ10} implies that  $tu_0 + (1-t)u_1 \in \mathcal E_\chi(X,\theta)$.
\end{proof}

If $f$ is smooth, it follows from a balayage argument that  $\theta_{P_{\theta}(f)}^n$ is concentrated on the contact set $\{P_{\theta}(f)=f\}$. Using the capacity theory developed in the previous subsection, the result also holds for more general functions $f$:

\begin{prop}\label{prop: BD vanish general}
        Assume that $\{ \theta\}$ is big, $P_{\theta}(f)\neq -\infty$ and $f$ is  quasi-continuous and usc on $X$.  Then $\theta_{P_{\theta}(f)}^n$ does not charge $\{P_{\theta}(f) <f\}$.
\end{prop}

\begin{proof}
Without loss of generality we can assume that $f\leq 0$. If $f$ is smooth then the result follows from a balayage argument (or directly from Theorem \ref{thm: weak BD}). To treat the general case, we approximate $f$ from above (by semicontinuity) by a sequence of smooth functions $(f_j)$.  We can also assume that $f_j\leq 0$. Set $\varphi_j:=P_{\theta}(f_j)$, $\varphi:=P_{\theta}(f)$ and note that $\varphi_j\searrow  \varphi$. For each $j\in \NN$ the measure $\theta_{\varphi_j}^n$ vanishes in the set $\{\varphi_j<f_j\}$. Now, we want to pass to the limit as $j\to +\infty$.  We first fix $k,l\in \NN$ and set 
        \[
        U_{k,l} = \{\varphi_k < f\} \cap \{\varphi > V_{\theta}-l\}. 
        \]
For any $j>k$, note that on $U_{k,l}$ we have $\varphi_j \geq \varphi>V_{\theta}-l$ and 
$\{\varphi_k < f\} \subset \{\varphi_j< f_j\}$.
It thus follows from definition of the non-pluripolar product (see \cite{BEGZ10}) that for any $j>k$, the measure $\theta_{\max(\varphi_j,V_{\theta}-l)}^n=\theta_{\varphi_j}^n$ vanishes on $U_{k,l}$. By assumption $f$ is quasi-continuous, hence $U_{k,l}$ is quasi open. More precisely, for any fixed $\vep>0$ there exists an open set $V_{\vep}$ such that the set  $G_{\vep}:=(V_{\vep} \setminus U_{k,l} \cup U_{k,l} \setminus V_{\vep})$ satisfies $\capo(G_{\vep})\leq \vep$. Observe that for fixed $l$ all the functions $\max(\varphi_j,V_{\theta}-l)$ are sandwiched between $V_\theta-l$ and $V_\theta$. It then follows that
        \[
        \sup_{j\in \NN} \int_{G_{\vep}} \theta_{\max(\varphi_j,V_{\theta}-l)}^n \leq A\capa_\theta (G_{\vep}) \leq A'\capo^{1/n}(G_{\vep}) \leq A'\vep^{1/n}, 
        \]
where the last inequality follows from the comparison of capacities in Theorem \ref{thm: comparison of capacity}.
Consequently, $\sup_{j\in \NN} \int_{V_{\vep}} \theta_{\max(\varphi_j,V_{\theta}-l)}^n  \leq A'\vep^{1/n}$, and 
the continuity of the Monge-Amp\`ere operator allows to take the limit, and we ultimately obtain: 
\[
\int_{U_{k,l}} \theta_{\max(\varphi,V_{\theta}-l)} ^n\leq C\vep^{1/n},  
\]
for some positive constant $C$ independent of $\vep$ (but dependent on $l$). Now letting $\vep \to 0$ we see that $\theta_{\max(\varphi,V_{\theta}-l)}^n$ vanishes in $U_{k,l}$. Letting $l\to +\infty$, and using the definition of the non-pluripolar product, we see that $\theta_\varphi^n$ vanishes in $\{\varphi_k<f\}$. Now, letting  $k\to +\infty$ we obtain the result. 
\end{proof}

\begin{theorem}
        \label{thm: Dravas criterion 1}
        Assume that $\psi, \varphi\in \Ec(X,\theta)$. Then $P_{[\theta,\varphi]}(\psi)=\psi$. 
\end{theorem}
\begin{proof}
        For each $t>0$, since $\min(\varphi+t,\psi)$ is usc and quasi-continuous, it follows from Proposition \ref{prop: BD vanish general} that $\theta_{\psi_t}^n$ vanishes on $\{\psi_t<\min(\varphi+t,\psi)\}$, where $\psi_t:=P_{\theta}(\min(\varphi+t,\psi))$. Because $\{\psi_t<\psi < \varphi +t\}\subset \{\psi_t<\min(\varphi+t,\psi)\}$ it thus follows that 
        \[
        \int_{\{\psi_t<\psi < \varphi +t\}} \theta_{ \psi_t}^n=0. 
        \]
        Fix $s>0,j>0$ and set $v:=P_{[\theta,\varphi]}(\psi)=(\lim_{t\to +\infty}\psi_t)^{*}$, $\psi_{t,j}:=\max(\psi_t,V_{\theta}-j)$, $v_j:=\max(v,V_{\theta}-j)$. It is clear that $\psi_{t,j} \nearrow v_j$ almost everywhere as $t\nearrow \infty$. By definition of the non-pluripolar Monge-Amp\`ere measure it follows that
        \[
        \int_{\{V_{\theta}-j<\psi_t<\psi<\varphi +t\}} \theta_{\psi_{t,j}}^n=0. 
        \]
        For $t>s$, we have $\psi_s \leq \psi_t \leq v \leq \psi$. Consequently, $\{V_{\theta}-j<\psi_s \leq v<\psi<\varphi +s\}\subset \{V_{\theta}-j<\psi_t<\psi<\varphi +t\}$ and we have
        \[
        \int_{\{V_{\theta}-j<\psi_s \leq v<\psi<\varphi +s\}} \theta_{\psi_{t,j}}^n=0. 
        \]
        Now, using the same trick as in the proof of Proposition \ref{prop: BD vanish general} we let $t\to +\infty$ and arrive at 
        \[
        \int_{\{V_{\theta}-j<\psi_s \leq v<\psi<\varphi +s\}} \theta_{v_j}^n =0.
        \] 
        Letting $s\to +\infty$, then $j\to +\infty$, and noting that $\theta_v^n$ is a non-pluripolar measure, we can conclude that $\theta_v^n$ vanishes on $\{v<\psi\}$. Finally, $v= P_{[\theta,\varphi]}(\psi)\in \Ec(X, \theta)$  by Theorem \ref{thm: sub-extension}, thus we can apply the domination principle (Proposition \ref{prop: domination principle}) to conclude the proof. 
\end{proof}

\section{Weak geodesics in big cohomology classes}\label{sect: weak geodesic}

\subsection{Berndtsson's construction}

We introduce a notion of weak geodesics in big cohomology classes mimicking  Berndtsson's construction in \cite[Section 2.2]{Bern}. Fix $\theta$ a smooth closed $(1,1)$-form such that $\{\theta\}$ is big and also fix $\varphi_0,\varphi_1 \in \psh(X,\theta)$ with minimal singularities. A \emph{subgeodesic} of $\varphi_0$, $\varphi_1$ is a curve $[0,1]\ni t\mapsto u_t\in \psh(X,\theta)$ such that
\begin{enumerate}
        \item[(i)] For each $t\in [0,1]$, the function $u_t$ has minimal singularities,
        \item[(ii)] $u_{0,1}=\lim_{t\to 0,1}u_t\leq \varphi_{0,1}$ pointwise, 
        \item[(iii)] The complexification  $X\times D\ni (x,z)\mapsto U(x,z):=u_{\log|z|}(x)$ is $\pi^*\theta$-psh on $X\times D$, where $D:=\{z\in \CC\setdef 1< |z|< e\}$ is the annulus in $\CC$ and $\pi: X \times D \to X$ is the trivial projection.  
\end{enumerate}

Let us mention that a curve $(\alpha,\beta) \in t \to u_t \in \psh(X,\theta)$ is a (general) subgeodesic if it satisfies only the appropriate version of (iii) above. In this case $\alpha,\beta$ may even take $\pm \infty$ as a value. We will not make a difference between the curve $t \to u_t$ and its complexification $U$. 
 
The \emph{weak geodesic} $[0,1] \ni t \to \varphi_t \in \psh(X,\theta)$ with minimal singularities joining $\varphi_0$ to $\varphi_1$ is defined as the envelope of all subgeodesics, i.e.,
\begin{equation}\label{eq: geoddef}
\varphi_l(x):= \sup \left\{u_l(x), \ \textrm{where }t \to u_t \textrm{ is a subgeodesic of} \ \varphi_0,\varphi_1\right\}, \ \ l \in [0,1], x \in X.
\end{equation}
\begin{lemma}
        \label{lem: boundary limit of weak geodesic}
        Let $t \to \varphi_t$ be the weak geodesic joining $\varphi_0,\varphi_1 \in \psh(X,\theta)$ with minimal singularities, constructed as above. Then there exists $C=C(\varphi_0,\varphi_1)>0$ such that 
        \begin{equation}
                \label{eq: boundary limit of weak geodesic}
                |\varphi_t -\varphi_{t'}| \leq C|t-t'|, \ \  t,t'\in [0,1].
        \end{equation}
Additionally, for the complexification $\Phi(x,z):= \varphi_{\log|z|}(x)$ we have 
\begin{equation}\label{eq: geodesicequation}
(\pi^*\theta +dd^c \Phi)^{n+1}=0 \textup{ on}\; \Amp(\{\theta\}) \times D,
\end{equation} 
where equality is understood in the weak sense of measures. 
\end{lemma}
\begin{proof}
The proof is essentially the same as in \cite[Section 2.2]{Bern}, so we only point out the ideas. Consider the following (barrier) subgeodesic of $\varphi_0,\varphi_1$: 
        \[
        u_t(x):=\max(\varphi_0-Ct, \varphi_1+C(t-1)),
        \] 
        where $C$ is a positive constant such that $\varphi_1-C\leq \varphi_0\leq \varphi_1+C$. It is clear that $[0,1]\ni t \to u_t \in \psh(X,\theta)$ is a subgeodesic of $\varphi_0,\varphi_1$, and by $t$-convexity of $t \to \varphi_t$ we can write $u_t \leq \varphi_t \leq (1-t)\varphi_0 + t\varphi_1$, hence the conclusion about Lipschitz continuity of $t \to \varphi_t$ follows. 

The proof of \eqref{eq: geodesicequation} follows from a standard balayage argument and we refer the interested reader to \cite[Section 2.2]{Bern}, to see how the ideas from there generalize to our setting.
\end{proof}

Next we prove a version of the comparison principle:
\begin{prop}
        \label{prop: comparison principle}
        Assume that $u,v\in \psh(X\times D,\pi^*\theta)$ satisfies $V_\theta -C \leq u_s,v_s, s \in D$ for some $C>0$. If $\liminf_{(x,z)\to X\times \partial D}(u-v)\geq 0$ then 
        \begin{equation*}
                \int_{\{u<v\}\cap \Amp(\{\theta\}) \times D} (\pi^*\theta +dd^c v)^{n+1}  \leq \int_{\{u<v\}\cap \Amp(\{\theta\}) \times D} (\pi^*\theta +dd^c u)^{n+1}.
        \end{equation*} 
\end{prop}
\begin{proof}
Fix $\vep>0$, $\delta>0$. As $\theta$ is big we can find $\psi\in \psh(X,\theta)$, $\sup_X \psi=0$, with analytic singularities such that $X\setminus \Amp(\{\theta\})= \{\psi=-\infty\}$ and $\psi \leq u_s,v_s, \ s \in D$, such that $\theta+dd^c \psi$ dominates a K\"ahler form. 
Consider 
        \[
        u_{\vep}:=\max(u,v_{\vep}), \ v_{\vep}:=(1-\delta)v+\delta\lambda \psi-\vep,
        \] 
for some constant $\lambda>1$. If $\lambda-1$ is small enough we have 
\[
\theta + dd^c \lambda \psi = \lambda (\theta+dd^c \psi) + (1-\lambda)\theta \geq \lambda\omega+(1-\lambda)\theta\geq 0,
\]
where $\omega$ is a K\"ahler form on $X$. Thus $v_{\vep}$ is $\pi^*\theta$-psh on $X\times D$. 
Observe that for some open relatively compact $\Omega'\Subset \Amp(\{\theta\})$ ($\Omega'$ depends on $\lambda$), $K\Subset D$, we have $u_{\vep}\equiv u$ in a neighbourhood containing $(X\setminus \Omega')\times (D\setminus K)$. It follows that $ \int_Y(\pi^*\theta +dd^c u)^{n+1}=\int_Y (\pi^*\theta +dd^c u_{\vep})^{n+1}$, where $Y:=\Omega'\times K$. Indeed, for any test function $0\leq \chi\in \Cc^{\infty}(Y)$ which is identically $1$ in an open neighborhood $U$ inside $Y$ such that $\{u<u_{\vep}\}\subset U$ we have 
\[
\int_Y \chi [ (\pi^*\theta +dd^c u)^{n+1}- (\pi^*\theta +dd^c u_\vep)^{n+1}] = \int_Y (u-u_{\vep})dd^c \chi \wedge T =0,
\]
where $T$ is a positive $(n,n)$-current. 
 
 Recall that for locally bounded $\pi^*\theta$-psh functions $\varphi,\psi$, the maximum principle for the complex Monge-Amp\`ere operator yields 
        \begin{equation*}
                {\bf 1}_{\{\varphi<\psi\}} (\pi^*\theta +dd^c \max(\varphi,\psi))^{n+1} ={\bf 1}_{\{\varphi<\psi\}} (\pi^*\theta +dd^c \psi)^{n+1}.
        \end{equation*}
Using the above facts we can write
        \begin{eqnarray*}
        \int_{\{u<v_{\vep}\}\cap \Amp(\{\theta\})\times D} (\pi^*\theta +dd^c v_{\vep})^{n+1} &=& \int_{\{u<v_{\vep}\}\cap \Amp(\{\theta\})\times D} (\pi^*\theta +dd^c u_{\vep})^{n+1}\\
        &=& \int_{Y} (\pi^*\theta +dd^c u_{\vep})^{n+1}  - \int_{\{u\geq v_{\vep}\}\cap Y}  (\pi^*\theta +dd^c u_{\vep})^{n+1}\\
        &\leq & \int_{Y} (\pi^*\theta +dd^c u)^{n+1}  - \int_{\{u > v_{\vep}\}\cap Y}  (\pi^*\theta +dd^c u_{\vep})^{n+1}\\
        &=& \int_{Y} (\pi^*\theta +dd^c u)^{n+1} - \int_{\{u > v_{\vep}\}\cap Y}  (\pi^*\theta +dd^c u)^{n+1}\\
        &\leq & \int_{\{u < (1-\delta)v+\delta\lambda  \psi -\varepsilon\}\cap \Amp(\{\theta\}) \times D}  (\pi^*\theta +dd^c u)^{n+1}.
        \end{eqnarray*}
The left-hand side in the above estimate can be further estimated in the following way: 
\[
\int_{\{u<v_{\vep}\}\cap \Amp(\{\theta\})\times D} (\pi^*\theta +dd^c v_{\vep})^{n+1}  \geq (1-\delta)^n \int_{\{u<v_{\vep}\}\cap  \Amp(\{\theta\})\times D} (\pi^*\theta +dd^c v)^{n+1}. 
\]
Letting $\delta,\vep\to 0$, the proof is finished.
\end{proof}

\begin{prop}
        \label{prop: uniqueness}
        Let $\varphi_0,\varphi_1\in \psh(X,\theta)$ with minimal singularities and 
        $[0,1]\ni t \to u_t$ be a subgeodesic of 
        $\varphi_0,\varphi_1$ with minimal singularities 
        satisfying \eqref{eq: boundary limit of weak geodesic} and $u_{0,1} = \varphi_{0,1}$. 
        If the complexification satisfies $(\pi^*\theta +dd^c U)^{n+1}=0$ on
        $\Amp(\{\theta\})\times D$ then $t \to u_t$ is the weak geodesic connecting $\varphi_0$ to $\varphi_1$. 
\end{prop}
\begin{proof}
        Let $t \to \varphi_t$ be the weak geodesic connecting $\varphi_0,\varphi_1$. It suffices to show that $\Phi \leq U$. Fix $\rho$  a smooth function in $D$ such that $\rho=0$ on the boundary and $dd^c \rho =\sqrt{-1}dz\wedge d\bar{z}$. Fix also $\tau \in \psh(X,\theta)$ with minimal singularities such that $(\theta +dd^c \tau)^n=c\omega^n$, for some positive normalization constant $c$ and some fixed K\"ahler form $\omega$ on $X$. Observe that such a potential exists thanks to \cite[Theorem 4.1]{BEGZ10}. We normalize $\tau$ so that $\tau\leq \min(\varphi_0,\varphi_1)$. Applying the comparison principle (Proposition \ref{prop: comparison principle}) for $\Phi_{\vep}:=(1-\vep)\Phi+\vep(\rho+\tau)$ and $U$ we get
        \[
        \int_{\{U<\Phi_{\vep}\}\cap \Amp(\{\theta\}) \times D} ( \pi^*\theta + dd^c\Phi_{\vep})^{n+1} \leq \int_{\{U<\Phi_{\vep}\}\cap \Amp(\{\theta\}) \times D} ( \pi^*\theta + dd^cU)^{n+1}=0.
        \]
By expanding $( \pi^*\theta + dd^c\Phi_{\vep})^{n+1}\geq \vep^n(\theta +dd^c \tau)^n\wedge dd^c \rho =\vep^n c \omega^n \wedge dd^c \rho>0$, the inequality above gives that $U\geq \Phi_{\vep}$ almost everywhere, hence everywhere in $X\times D$. Letting $\vep \to 0$ we get the desired result. 
\end{proof}

In the K\"ahler case, geodesics joining smooth potentials are $C^{1,\bar{1}}$-smooth and the Monge--Amp\`ere energy is linear along these geodesics.  
When $\theta$ is also nef,  by adapting the proof of Chen \cite{Che00} (see \cite{Bou12}), we expect that a similar regularity result still holds in the ample locus of $\{\theta\}$. 

Recall that for potentials with minimal singularities the Monge--Amp\`ere energy is defined by the following expression:
\[
\AM(u) := \frac{1}{(n+1)\vol(\theta)} \sum_{j=0}^n \int_X (u-V_{\theta})\theta_u^j\wedge \theta_{V_\theta}^{n-j}.
\]
We proceed now to show that, as perhaps expected, the Monge--Amp\`ere energy is linear/convex along geodesics/subgeodesics with minimal singularities. However to argue this, we will need to use a careful mollification argument for subgeodesics in the time variable that will be detailed in the next subsection.  

\subsection{Approximation of subgeodesics}\label{sect: approx. of sub-geo}
Unless specified otherwise, assume for this subsection that $(0,1) \ni t \to \varphi_t \in \psh(X,\theta)$ is a subgeodesic with minimal singularities such that $|\varphi_t-\varphi_{t'}|\leq C|t-t'|$, for some positive constant $C$. Consider a smoothing kernel $\chi: \RR \rightarrow [0,1]$ supported in $[-1,1]$ with $\int_{\RR}\chi(t)dt=1$ and $\chi(t)=\chi(-t)$. We then set $\chi_{\vep}(t):= \frac{1}{\vep} \chi(t/\vep)$, so that the support of $\chi_{\vep}$ is $[-\vep,\vep]$ and $\int_{\RR} \chi_{\vep}(t)dt=1$. For each $\vep>0$ we consider 
\begin{equation}\label{eq: moll_approx_def}
\varphi_{\vep,t}(x):= \int_{0}^1 \varphi_{s}(x) \chi_{\vep}(t-s) ds,
\end{equation}
which is well defined for $t\in (\vep, 1-\vep)$. 

\begin{lemma}
For each $\vep>0$, $(\vep,1-\vep) \ni t \to \varphi_{\vep,t} \in \psh(X,\theta)$ is a $t$-Lipschitz subgeodesic with minimal singularities.      
\end{lemma}
\begin{proof}
        The fact that $t \to \varphi_{\vep,t}$ is a subgeodesic is trivial. Because $t \to \varphi_t$ is $t$-Lipschitz, it follows that $\varphi_{\vep,t}$ has minimial singularities for all $t \in [\vep,1-\vep]$. Working directly with \eqref{eq: moll_approx_def} we obtain $|\varphi_{\vep,t}-\varphi_{\vep,t'}|\leq C|t-t'|$.
\end{proof}
The next observation is simple but will be crucial for our approximation procedure. 

\begin{lemma}
        \label{lem: difference of quasi-psh}
        There exists $A_{\vep},B_{\vep}>0$ (that may blow up as $\vep \to 0$) such that $\dot{\varphi}_{\vep,t}/A_{\vep}$, $\ddot{\varphi}_{\vep,t}/B_{\vep}$ can be written as a difference of $\theta$-psh functions with minimal singularities $t \in (\vep,1-\vep)$. 
\end{lemma}
\begin{proof}
        Write $\chi_{\vep}'(t) = \rho^+_{\vep}(t)-\rho^-_{\vep}(t)$ where $\rho^+_{\vep},\rho^-_{\vep}$ are the positive and negative parts of $\chi_{\vep}'(t)$. They are clearly bounded (but the bound blows up as $\vep\to 0$). Now, let 
        \[
        u^{\pm}_{\vep,t}(x):= \int_0^1 \varphi_s(x)\rho^{\pm}_{\vep}(t-s) ds.
        \]
        It follows that 
        \[
        dd^c u^{\pm}_{\vep,t} =  \int_0^1 (dd^c\varphi_s) \rho^{\pm}_{\vep}(t-s) ds\geq -A_{\vep} \theta,
        \]
where $A_{\vep}= \int_{-\vep}^{\vep} \rho^+_{\vep}=\int_{-\vep}^{\vep} \rho^-_{\vep}$.
Observe indeed that $\int_{-\vep}^\vep \rho^+_{\vep}- \rho^-_{\vep} =\int_{-\vep}^\vep \chi'_\vep= 0$. 
To see that $u_{\vep,t}^{\pm}/A_\vep$ has minimal singularities we note that $|\varphi_{t}-\varphi_{t'}|\leq C|t-t'|$, hence 
        \[
        u^{\pm}_{\vep,t}(x):= \int_{-\vep}^\vep \varphi_{t-s}(x)\rho^{\pm}_{\vep}(s) ds \leq \int_{-\vep}^\vep \varphi_{t'-s}(x)\rho^{\pm}_{\vep}(s) ds + CA_{\vep}|t-t'|.
        \]
For $\ddot \varphi_{\vep,t}$, a similar argument works with the choice  $B_{\vep}:=\int_{-\vep}^\vep\max(\chi''_\vep(t),0)dt$.
\end{proof}
\begin{lemma}\label{lem: AM is differentiable along approximants}
        The Monge-Amp\`ere energy is twice differentiable along $t \to \varphi_{\vep,t}$ and the derivatives are given by
 \begin{equation}\label{eq: second derivative AM big 1}
\frac{d}{dt} \AM(\varphi_{\vep,t}) = \int_X \dot{\varphi}_{\vep,t} \theta_{\varphi_{\vep,t}}^n, \       \frac{d^2}{dt^2} \AM(\varphi_{\vep,t}) = \int_X \ddot{\varphi}_{\vep,t} \theta_{\varphi_{\vep,t}}^n - n\int_X d\dot{\varphi}_{\vep,t}\wedge d^c \dot{\varphi}_{\vep,t}\wedge \theta_{\varphi_t}^{n-1}.
        \end{equation}
\end{lemma}
\begin{proof}
        For notational convenience we remove the subscript $\vep$. By basic properties of the $\AM$ functional (see \cite{BEGZ10,BB10}) we have, for $t\in (\vep,1-\vep), s>0$, 
        \[
        \int_X (\varphi_{t+s}-\varphi_t)\theta_{\varphi_{t+s}}^n \leq \AM(\varphi_{t+s})- \AM(\varphi_t) \leq  \int_X (\varphi_{t+s}-\varphi_t) \theta_{\varphi_t}^n.
        \] 
        Dividing by $s>0$ the right-hand side then converges (as $s\to 0$)  to $\int_X \dot{\varphi}_t\theta_{\varphi_t}^n$. The left-hand side can be estimated by 
        \[
        \frac{1}{s}\int_X (\varphi_{t+s}-\varphi_t)\theta_{\varphi_{t+s}}^n\geq \int_X \dot{\varphi}_t \theta_{\varphi_{t+s}}^n. 
        \]
 As $\dot{\varphi}_t$ is bounded and quasi-continuous (with respect to the Monge-Amp\`ere capacity $\capo$), Corollary \ref{cor: convergence big} allows to pass to the limit as $s\to 0$ and the differentiability of $\AM(\varphi_t)$ follows. 
 
 To prove the formula for the second derivative we fix $t\in (\vep,1-\vep)$ and $s>0$ small enough and prove that 
\begin{equation*}
\lim_{s\to 0^+}d(s):=\lim_{s\to 0^+} \frac{1}{s} \left( \int_X \dot{\varphi}_{t+s} \theta_{\varphi_{t+s}}^n - \dot{\varphi}_t \theta_{\varphi_t}^n\right)  
\end{equation*}
 equals to the right-hand side of \eqref{eq: second derivative AM big 1}. The same formula for the left limit will then follow automatically, as we can always "reverse" the time direction of a subgeodesic. Setting $T_s:= \sum_{j=1}^{n-1} \theta_{\varphi_t}^j \wedge \theta_{\varphi_{t+s}}^{n-1-j}$, we can write
\[
d(s)= \frac{1}{s} \left( \int_X \dot{\varphi}_{t+s} dd^c (\varphi_{t+s}-\varphi_t) \wedge T_s + \int_X (\dot{\varphi}_{t+s}-\dot{\varphi}_t)\theta_{\varphi_t}^n\right).
\]
Thanks to Theorem 1.14 in \cite{BEGZ10} we can integrate by parts in the first term and obtain
\begin{equation}
        \label{eq: second derivative AM big 3}
d(s)= \frac{1}{s} \left( \int_X (\varphi_{t+s}-\varphi_t) dd^c\dot{\varphi}_{t+s}  \wedge T_s + \int_X (\dot{\varphi}_{t+s}-\dot{\varphi}_t)\theta_{\varphi_t}^n\right). 
\end{equation}
As $s\to 0$, $(\varphi_{t+s}-\varphi_t)/s$ decreases to $\dot{\varphi}_t$, all of them being quasi-continuous and uniformly bounded. On the other hand, by Lemma \ref{lem: difference of quasi-psh} we can write $\dot{\varphi}_{t+s}/A_{\vep}=u^+_{t+s}-u^{-}_{t+s}$, where  $u^{\pm}_{t+s}$ are $\theta$-psh functions with minimal singularities that converge uniformly to $u^{\pm}_t$ as $s\to 0$. By using Corollary \ref{cor: convergence big}, the first term in the right-hand side of \eqref{eq: second derivative AM big 3} converges to  $n\int_X \dot{\varphi}_t dd^c \dot{\varphi}_t\wedge \theta_{\varphi_t}^{n-1}$. Moreover, using dominated convergence for the second term we obtain
$$\lim_{s\to 0^+} \frac{1}{s}\int_X (\dot{\varphi}_{t+s}-\dot{\varphi}_t)\theta_{\varphi_t}^n = \int_X \ddot{\varphi}_{t}\theta_{\varphi_t}^n.$$
Observe in fact that $\ddot{\varphi}_{t}$ is uniformly bounded from above as it can be written as the difference of two qpsh functions. Finally, an integration by parts gives
\[
\lim_{s\to 0^+}d(s)=-n\int_X d\dot{\varphi}_{t}\wedge d^c \dot{\varphi}_{t}  \wedge \theta_{\varphi_t}^{n-1} + \int_X \ddot{\varphi}_{t}\theta_{\varphi_t}^n. 
\]
\end{proof}

\begin{coro}\label{cor: AM convex approximant}
The Monge-Amp\`ere energy is convex along $t \to \varphi_{\vep,t}$.     
\end{coro}
\begin{proof}We again drop the $\varepsilon$ subscript and fix $t \in (\vep,1-\vep)$ for the duration of the proof. In view of Lemma \ref{lem: AM is differentiable along approximants} it is enough to prove that 
\begin{equation}\label{eq: proof convex AM big approximant}
-n d\dot{\varphi}_{t}\wedge d^c \dot{\varphi}_{t}  \wedge \theta_{\varphi_t}^{n-1} +  \ddot{\varphi}_{t}\theta_{\varphi_t}^n\geq 0
\end{equation}
in the weak sense of measures in $\Amp(\{\theta\})$. This property is local, hence we can work in relatively compact open subset $K$ of $\Amp(\{\theta\})$, and approximate $\varphi_t(x)$ using a convolution in the space variable $x$. As we show now, since $\dot{\varphi}_t$, $\ddot{\varphi}_t$ are bounded in $K$ and can be written as the difference of two quasi-psh functions with minimal singularities, the convergence of the appropriate approximating measures to the left-hand side in \eqref{eq: proof convex AM big approximant} follows from standard pluripotential theory. Indeed, fix a coordinate ball $B \subset K$. We will show that \eqref{eq: proof convex AM big approximant} holds in $B$ in the weak sense of measures.

 We can assume existence of a smooth local potential $\tau \in C^\infty(B)$ such that $\theta=dd^c \tau$. Let $\tilde \rho_{\delta}(x)$ be a smoothing kernel in $\mathbb{C}^n$ and consider 
\[
\tilde \varphi_t^{\delta}(x):=\int_{B}( \tau(x-\zeta)+\varphi_t(x-\zeta)) \tilde \rho_{\delta}(\zeta) dV(\zeta) - \tau(x).
\] 
 Since the complexification of $\tilde \varphi_t^{\delta}$ is smooth and $\pi^*\theta$-psh in $B_\delta\times D_{\vep}$, it follows that \eqref{eq: proof convex AM big approximant} holds for $\tilde \varphi_t^{\delta}$ and $\tilde \varphi_t^{\delta} \searrow \varphi_t$.

 By Lemma \ref{lem: difference of quasi-psh} we can write $\dot{\varphi}_t:=u_+-u_{-}$, where $u_+,u_-$ are bounded quasi psh functions on $B_\delta$. Then the corresponding smooth quasi psh functions $v^\delta_\pm$, defined by 
$$
v^\delta_{\pm}(x):= \int_0^1\int_{B} ( \tau(x-\zeta)+\varphi_s(x-\zeta)) \tilde \rho_{\delta}(\zeta)\,\rho_\vep^{\pm} (t-s) ds \,dV(\zeta),
$$
converge decreasingly to $u_{\pm}$ and we have $\dot {\tilde \varphi}_t^\delta = v^\delta_{+}-v^\delta_{-}$. A similar thing is true for the second derivatives $\ddot{\varphi}_t,\ddot{\tilde \varphi}^\delta_t$ as well.  As all the functions involved are quasi psh and bounded on $B_\delta$, by Bedford--Taylor theory, positivity in \eqref{eq: proof convex AM big approximant} is preserved as we take the limit $\delta \to 0$, and  \eqref{eq: proof convex AM big approximant} holds restricted to $B_\delta$, finishing the proof.  
\end{proof}

\begin{theorem}\label{thm: AM convex big}
        Assume that $\{\theta\}$ is big and $(0,1) \ni t \to \varphi_t \in \psh(X,\theta)$ is a subgeodesic with minimal singularities. Then the Monge-Amp\`ere energy $\AM$ is convex along $t \to \varphi_t$. 
\end{theorem}
\begin{proof} Fix $\vep >0$ for the moment. As it suffices to show convexity of $ t \to \AM(\varphi_t)$ on $(\vep,1-\vep)$, and $t \to \varphi_t$ has minimal singularity potentials, without loss of generality we can assume that $t \to \varphi_t$ is Lipschitz in $t$ and let $C>0$ be such that $|\varphi_t-\varphi_{t'}|\leq C|t-t'|$.  Let $\varphi_{\vep,t}$ be the subgeodesics approximating $t \to \varphi_t$ constructed above. Since $\AM(\varphi_{\vep,t})\to \AM(\varphi_t)$ as $\vep\to 0$ \cite[Proposition 4.3]{BB10}, it suffices to prove the convexity of $t \to \AM(\varphi_{\vep,t})$. But this follows from Corollary \ref{cor: AM convex approximant}. 
\end{proof}

\subsection{Linearity of the Monge-Amp\`ere energy along geodesics}
The regularization technique in the Subsection \ref{sect: approx. of sub-geo} can also be used to prove linearity of Monge--Amp\`ere energy along geodesics in big classes.  Although this result is not used in this paper we think it will be useful in the future. 

Assume $\theta$ is a smooth closed $(1,1)$-form whose cohomology class is big.
Fix $\varphi_0,\varphi_1$ two $\theta$-psh functions on $X$ with minimal singularities. The complex plane $\CC$ is now viewed as a piece of $\CC\PP^1$ (so that $\CC\PP^1=\CC \cup H_{\infty}$) equipped with the Fubini-Study metric $\omega_{FS}$. Accodringly, the product $M:=X\times \CC\PP^1$ is equipped with a smooth $(1,1)$-form $\Theta:=\pi_1^*\theta +\pi_2^* \omega_{FS}$. We use the following change of coordinates between $\CC\PP^1$ and $\CC$:
\[
\CC\PP^1\setminus H_{\infty} \ni [z_1:z_2] \mapsto \frac{z_1}{z_2} \in \CC,
\]
where $H_{\infty}:=\{z_2=0\}$.

\begin{lemma}
        \label{lem: minimal sing product}
One has $V_{\Theta}(x,z):=V_{\theta}(x)$ for all $z\in \CC\PP^1$. 
\end{lemma}

\begin{proof}
By definition 
$$V_{\Theta}(x,z):= \sup\{U(x,z)\quad \Theta-{\rm psh}\;:\; U\leq 0 \;  {\rm on\; }M\}.$$
Clearly $V_{\Theta}(x,z)\geq V_{\theta}(x)$ since $V_{\theta}(x)$ is  a candidate in the envelope. Moreover, we observe that for each $z\in \CC\PP^1$, $V_{\Theta}(x,z)$ is a $\theta$-psh function and $V_{\Theta}(x,z)\leq0$ on $X$, thus $V_{\Theta}(x,z)\leq V_\theta (x)$. Hence the conclusion.
\end{proof}

\begin{lemma}
        \label{lem: harmonicity}
        Let $F\in \Cc^{\infty}(M,\RR)$ be a smooth function on $M$ which is $S^1$-invariant when restricted to $X\times \CC$.  Denote by $\Phi:=P_{\Theta}(F)$ the Monge-Amp\`ere envelope on $M$ of $F$ with respect to $\Theta$. Then the function $
        \CC\PP^1 \ni z \mapsto G(z):= \AM(\Phi(\cdot,z)) $
        satisfies 
        \begin{equation}
                \label{eq: harmonic M}
                \omega_{FS}+dd_z^c G = (\pi_2)_\star (\Theta +dd_{x,z}^c \Phi)^{n+1},
        \end{equation}
        in the sense of currents. 
\end{lemma}

\begin{proof}
        As $H_{\infty}$ is pluripolar in $\CC\PP^1$ it suffices to prove \eqref{eq: harmonic M} in $\CC$.  As $F$ is $S^1$-invariant it follows that $\Phi(x,z)$ is also $S^1$-invariant in $z$. Using convolution as in Subsection \ref{sect: approx. of sub-geo} we denote by $\Phi_{\vep}$ the approximants, i.e. 
\[
\Phi_{\vep}(x,z) := \int_{\CC} \Phi(x,\zeta) \chi_{\vep}(|z-\zeta|^2)dV(\zeta),
\]      
where $\chi_{\vep}$ is a family of smoothing kernels.   For each $\vep>0$, $\Phi_{\vep}$ is $\Theta_{\vep}$-psh on $M$ where $\Theta_{\vep}:=\pi_1^*\theta +\gamma_{\vep}\pi_2^* \omega_{FS}$, with $\gamma_{\vep}>1$ decreasing to $1$. Indeed, 
\[
        dd^c_{x,z} \Phi\star \chi_{\vep} = dd^c \Phi \star \chi_{\vep} \geq (-\theta  -\omega_{FS})\star \chi_{\vep} \geq -\theta -\omega_{FS}\star \chi_{\vep}. 
        \]      
Moreover, since $\omega_{FS}$ is smooth on $\CC\PP^1$, we have
        \[|\omega_{FS} -\omega_{FS}\star \chi_{\vep}| \leq \int_{\CC}|\omega_{FS}(z)-\omega_{FS}(z-\zeta)|\chi_\vep(|\zeta|^2) dV(\zeta) \leq C \int_{\CC}|\zeta|\,\chi_\vep(|z'|^2) dV(z')\leq C \vep . \]
This means that we can find $\gamma_\vep$ decreasing to $1$ such that $\omega_{FS}\star \chi_\vep \leq \gamma_\vep \omega_{FS}$.
 It can be shown in the same way as in Subsection \ref{sect: approx. of sub-geo} that $G_{\vep}(z):=\AM(\Phi_{\vep}(\cdot, z))$ is smooth and its complex Hessian is given by 
        \begin{equation}\label{eq: derivative AM epsilon}
        \gamma_{\vep}\omega_{FS} + dd_z^c G_{\vep}(z) = (\pi_2)_\star (\Theta_{\vep} +dd_{x,z}^c \Phi_{\vep})^{n+1}.
        \end{equation}
        To prove this we first explain how to compute all the second order partial derivatives of $G_{\vep}$. Fix $\xi\in \CC=\RR^2$ a unit vector  and denote by $\partial_{\xi} f(z):=\lim_{h\to 0^+} (f(z+h\xi)-f(z))/h$ the derivative of $f$ in the direction $\xi$. Note that 
        \[
        \frac{1}{h} \left[\chi_{\vep}(|z+h\xi|^2)-\chi_{\vep}(|z|^2)\right] \rightarrow \partial_{\xi}(\chi_{\vep}(|z|^2))
        \]
        as $h\to 0$ uniformly in $z$. It follows that 
        \[
        \frac{1}{h} \left[\Phi_{\vep}(x,z+h\xi)-\Phi_{\vep}(x,z)\right] \rightarrow \partial_{\xi}\Phi_{\vep}(z)
        \]
        as $h\to 0$ uniformly in $x,z$. By the same arguments as in Subsection \ref{sect: approx. of sub-geo} the first and second partial derivatives of $z\mapsto \Phi_{\vep}(x,z)$ can be written as $C_{\vep}(u^+ -u^{-})$ where $u^{\pm}$ are $\theta$-psh functions with minimal singularities. They are thus uniformly bounded by a constant $C'_{\vep}$ (which blows up as $\vep\to 0$). Thus the same arguments as in Subsection \ref{sect: approx. of sub-geo} show that $z\mapsto G_{\vep}(z)$ is twice differentiable (even smooth). Set $g_{\vep}=\gamma_{\vep}\log (1+|z|^2)$, where $g=\log (1+|z|^2)$ is the potential of $\omega_{FS}$ in $\CC$. The second derivative $\partial_z \partial_{\bar{z}}$ of $G_{\vep}+g_{\vep}$ is given by 
\[
\partial_z \partial_{\bar{z}} (G_{\vep}+g_{\vep}) = \int_X \partial_z \partial_{\bar{z}} (\Phi_{\vep}+g_{\vep}) (\theta + dd^c_x \Phi_{\vep})^n -n d_x \partial_z(\Phi_{\vep}+g_{\vep}) \wedge d^c_x \partial_{\bar{z}} (\Phi_{\vep}+g_{\vep}) \wedge (\theta + dd^c_x \Phi_{\vep})^{n-1}.
\]

Let $H(z)$ denote the integrand in the right-hand side which is a positive measure on $X$.  Indeed, since it is a local property we can argue locally and use an approximation technique as in the proof of Corollary \ref{cor: AM convex approximant}. Moreover, we infer that $\frac{\sqrt{-1}}{\pi}H(z)dz\wedge d\bar{z}$  is the Monge-Amp\`ere measure of $\Phi_{\vep}+g_{\vep}$ with respect to the form $\pi_1^* \theta$, i.e. $(\pi_1^*\theta +dd^c (\Phi_{\vep}+g_{\vep}))^{n+1}$. This together with
\begin{eqnarray*}
                \gamma_{\vep}\omega_{FS} + dd_z^c G_{\vep}(z) =  dd_z^c (g_{\vep} + G_{\vep}) = dd_z^c \AM(\Phi_{\vep}(\cdot, z) + g_{\vep}(z))
\end{eqnarray*}
justify the formula \eqref{eq: derivative AM epsilon}. 

Fix $\chi: \CC\rightarrow \RR$ a smooth function with compact support in $\CC$. We want to prove that 
        \[
        \int_{\CC} \chi (\omega_{FS} + dd^c G) = \int_{M} \chi (\Theta +dd^c \Phi)^{n+1}.
        \] 
The above formula is true for the approximants $G_{\vep}$ and $\Theta_{\vep}$ as  we discussed above. Now we explain how we can insure the convergence when we take the limit in \eqref{eq: derivative AM epsilon} as $\vep\to 0$.
 To deal  with the left-hand side we prove that $G_{\vep}$ converge pointwise to $G$.  For fixed $z\in \CC$ we can find constants $c_{\vep}$ converging to $0$ such that $\Phi_{\vep} +c_{\vep}$ decreases to $\Phi$. Then $G_{\vep}(z)$ converges to $G(z)$ by basic properties of the $\AM$ functional \cite[Proposition 4.3]{BB10}. As $G_{\vep}$ is uniformly bounded independent of  $\vep$, the convergence of the current $dd^c G_{\vep}$ follows.  Now we treat the convergence of the right-hand side of \eqref{eq: derivative AM epsilon}. Fix a K\"ahler form $\omega_M$ on $M$ and $\delta>0$. Let $U$ be an open  neighbourhood of the pluripolar set $E= (X\setminus \Amp(\{\theta)\})\times H_{\infty}$ such that ${\rm Cap}_{\omega_M}(U)<\delta$. Note that Theorem \ref{thm: comparison of capacity} gives that
  \[
  {\rm Cap}_{\Theta_{\vep}}(U) \leq C {\rm Cap}_{\omega_M}(U) ^{1/n} =O(\delta^{1/n}),
  \]
  where $C$ is independent of $\vep$. And, by Bedford-Taylor theory \cite{BT82} the Monge-Amp\`ere measure $(\Theta_{\vep}+ dd^c \Phi_{\vep})^{n+1}$ converges to $(\Theta+dd^c \Phi)^{n+1}$ in $M\setminus U$.  Letting $\vep\to 0$ and then $\delta\to 0$ we arrive at the conclusion.
\end{proof}

We move forward to proving the linearity of the $\AM$ functional. Consider $(M,\Theta)$ as above and let $\rho:D\rightarrow \RR$ be a smooth potential of $\omega_{FS}$ in $D\subset \CC\PP^1$ with zero boundary values.

\begin{lemma}
        \label{lem: linear}
        Assume that $F$ is a smooth function on $M$ which is $S^1$-invariant in $X\times \CC$ in the variable $z$. Let $\varphi_t$ be the geodesic connecting $\varphi_0=P_{\theta}(F(\cdot, 1))$ to $\varphi_1=P_{\theta}(F(\cdot, e))$. Let $\Phi$ be the envelope on $M$ of $F$ with respect to $\Theta$. If $F(x,z)+ \rho(z)> \varphi_{\log|z|}(x)$ in $X\times D$ then $z\mapsto \AM(\Phi(\cdot, z)+\rho(z))$ is harmonic in $D$. 
\end{lemma}
\begin{proof}
When restricted to $X\times D$, we have $\Phi +\rho$ is $\pi_1^* \theta$-psh and has boundary values $\varphi_{0,1}$, thus by definition, $\Phi  \leq \varphi_t-\rho<F$.  It follows from \cite{BT82}  that $(\Theta +dd^c \Phi)^{n+1}=0$ in $X\times D$, which in turn implies that $\AM(\Phi(\cdot,z)+\rho(z))$ is harmonic in $D$ thanks to Lemma \ref{lem: harmonicity}.
\end{proof}

\begin{theorem}\label{thm: AM linear big}
        The $\AM$ energy is linear along weak geodesics with minimal singularities. 
\end{theorem}
\begin{proof}
        Fix $f_0,f_1$ two smooth functions in $X$  and denote by $\varphi_i=P_{\theta}(f_i), i=0,1$ the envelopes of $f_0,f_1$ respectively.  Let $\varphi_t$ be the geodesic connecting $\varphi_0,\varphi_1$. Observe also that by approximating any two given potentials with minimal singularities with a sequence of smooth functions, it suffices to prove linearity of $\AM$ along $\varphi_t$.

Let $F$ be a function on $M$ which is $S^1$-invariant in $X\times \CC$ in the variable $z$ and such that $F=f_{0,1}$ on $X\times \partial D$ and $+\infty$ elsewhere. Consider a sequence $(F_j)_{j}$ of smooth functions $F_j\uparrow F$, which are also $S^1$-invariant in $X\times \CC$ in the variable $z$ and such that $F_j + \rho >\varphi_{\log |z|}(x)$ in $X\times D$. Let $\Phi_j$ be the envelope on $M$ of $F_j$ with respect to $\Theta$. Then $(\Theta+dd^c\Phi_j)^{n+1}$ is supported on 
$\{\Phi_j=F_j\}$. As $X\times \partial D$ is non-pluripolar in $M$ it follows that $\Phi_j$ is uniformly bounded from above. Thus $\Phi_j$ converges increasingly (almost everywhere) to $\Phi$ a $\Theta$-psh function with minimal singularities. We claim that 
\begin{equation}\label{vanishing}
\int_{\{\Phi<P_{\Theta}(F)\}}(\Theta+dd^c \Phi)^{n+1}=0.
\end{equation}
Too see this we observe that the Monge-Amp\`ere measure of $\Phi$ is concentrated on $X\times \partial D$ and that$ \{\Phi<P_{\Theta}(F)\}\subset X\times (\CC\PP^1\setminus \partial D)$ since $\Phi=P_{\Theta}(F)=\varphi_{0,1}$ on $X\times \partial D$. Indeed, on any open relatively compact subset $K$ of $X\times (\CC\PP^1\setminus \partial D)$  one has that $\Phi_j < F_j$ for $j$ large enough (since $F_j$ increases to $+\infty$ uniformly in $K$ and $\Phi_j$ is uniformly bounded from above). By  the continuity property of the complex Monge-Amp\`ere operator we get that $(\Theta+dd^c \Phi)^{n+1}(K)=0$. It follows from (\ref{vanishing}) and the domination principle \cite[Corollary 2.5]{BEGZ10} that $\Phi=P_{\Theta}(F)$.

Now, we claim that $\Phi+\rho=\varphi_{\log |z|}$ in $X\times D$. Indeed, consider 
\[
\begin{cases}
U_0(x,z):= \varphi_0(x)+ A(\log (|z|^2+3)-\log(|z|^2+1) -\log 2); \\ U_1(x,z):=\varphi_1(x)+ A(\log |z|^2-\log(|z|^2+1) +\log (e^2+1) -2).
\end{cases}
\] 
For $A>0$ big enough  $U:=\max(U_0,U_1) = \varphi_{0,1}$ on $\partial D$
%
and it is $\Theta_A$-psh, where $\theta_A:=\pi_1^*\theta +A\pi_2^*\omega_{FS}$. So we can apply our previous analysis with this $(1,1)$-form $\Theta_A$ instead of $\Theta$. By definition of envelope we have $\Phi\geq U$ and in particular $\Phi \geq \varphi_{0,1}$ on $X\times \partial D$. Moreover, for each $z\in \partial D$, $\Phi(\cdot, z)$ is $\theta$-psh and dominated by $F=f_{0,1}$. It then follows that $\Phi\leq \varphi_{0,1}$ on $X\times \partial D$, giving the equality on the boundary. Furthermore, it follows from the proof of Lemma \ref{lem: linear} and the continuity of the Monge-Amp\`ere operator that the Monge-Amp\`ere measure $(\pi_1^*\theta +dd^c \rho +dd^c \Phi)^{n+1}$ vanishes in $X\times D$. Proposition \ref{prop: uniqueness} thus insures that $\Phi+\rho$ is the unique weak geodesic with minimal singularities joining $\varphi_0,\varphi_1$. Hence the claim.

Now, thanks to Lemma \ref{lem: linear} and \cite[Proposition 4.3]{BB10}, we know that $\AM(\Phi(\cdot, z)+\rho(z))= \AM(\varphi_{\log |z|})$ is harmonic in $D$ (and $S^1$-invariant!). Hence, $\AM$ is linear along $\varphi_t$, with $t=\log|z|$. This is what we wanted.  
\end{proof}

\subsection{Geodesic rays and the proof of Theorem \ref{thm: characterization E big classes}}

Given $\varphi\geq \psi$ two $\theta$-psh functions such that $\varphi$ has minimal singularities, we define the weak geodesic ray attached to $\varphi,\psi$ in the following way (see \cite{Dar13} for the K\"ahler case). 
For fixed $l>0$, we denote by $[0,l] \ni t \to u^{l}_t \in \psh(X,\theta)$ the weak geodesic segment joining $\varphi$ and $\max(\varphi-l,\psi)$. The same argument as in \cite[Lemma 4.2]{Dar13} shows that $u^{l}$ forms an increasing family of weak geodesics and we can then define the limit subgeodesic ray:
\begin{equation}\label{eq: v_t_def}
v(\varphi,\psi)_t := \left(\lim_{l\to +\infty} u^{l}_t \right)^*, \ t\in [0,+\infty).
\end{equation}
\begin{lemma}
        \label{lem: v is geodesic}
Assume that $\theta$ is big. The subgeodesic ray $[0,\infty) \ni t \to v(\varphi,\psi)_t \in \psh(X,\theta)$ is a weak geodesic ray. 
\end{lemma}

\begin{proof}
As all complexifications $U^l \in \psh(X\times D_l,\pi^*\theta)$ satisfy the appropriate complex Monge--Amp\`ere equation on the domains $\Amp(\{\theta\}) \times D_l$ and are locally bounded there, it follows from continuity property of the complex Monge--Amp\`ere operator  that the complexification of $t \to v(\varphi,\psi)_t$ satisfies the homogeneous Monge-Amp\`ere equation on $\Amp(\{\theta\}) \times D_\infty$ as well. Since all the curves $t \to u^l_t$ are uniformly $t$-Lipschitz continuous, so is their limit $t \to v(\varphi,\psi)_t$, hence Proposition \ref{prop: uniqueness} gives that  $t \to v(\varphi,\psi)_t$ must be a weak geodesic ray, i.e., for any closed interval $I \subset [0,\infty)$, the restriction $I \ni t \to v(\varphi,\psi)_t \in \psh(X,\theta)$ is the weak geodesic joining the potentials corresponding to the endpoints of $I$.
\end{proof}

Now we introduce an invariant of $\psi$. It is clear that $[0,\infty) \ni t \to \max(\varphi-t,\psi)$ is a subgeodesic ray with minimal singularities. Thus $t \to \AM(\max(\varphi-t,\psi))$ is convex by Theorem \ref{thm: AM convex big} and decreasing by \cite[Proposition 2.8]{BEGZ10}. This implies that the following limit is well defined: 
$$c_\psi := \lim_{t \to \infty} \frac{\AM(\max(\varphi-t,\psi)) - \AM(\varphi)}{t} \leq 0.$$
Recall the  cocycle formula of \cite[Corollary 3.2]{BB10}:
\begin{equation}\label{eq: AMcocycle}
\AM(u)-\AM(v)=\frac{1}{\textup{Vol}(\theta)(n+1)}\sum_{j=0}^n\int_X (u-v) \theta_u^{j} \wedge \theta_v^{n-j},
\end{equation}
where $u,v \in \psh(X,\theta)$ have minimal singularities. From this formula it follows that $c_\psi$ is independent of the choice potential with minimal singularities $\varphi$ satisfying $\varphi \geq \psi$. Finally, the following result gives an attractive characterization of potentials $\psi$ for which $c_\psi=0$.
\begin{prop}\label{prop: E_char_c_psi} Given $\psi \in \psh(X,\theta)$, we have $c_\psi =0$ if and only if $\psi \in \mathcal E(X,\theta)$.
\end{prop}
\begin{proof}The proof is an adaptation of the arguments of \cite[Theorem 2.5]{Dar13} to our more general setting. As $c_\psi$ only depends on $\psi$ it is enough to work with the special subgeodesic ray $t \to \psi_t:=\max(V_\theta-t,\psi)$. The cocycle formula \ref{eq: AMcocycle} implies that
$c_\psi=c_{\psi+c}$, thus we can assume that $\psi\leq V_\theta\leq 0$. By \cite[Proposition 2.8]{BBGZ13} 
$$
 \int_X (\psi_t-V_{\theta})\, \theta_{\psi_t}^n \leq \AM(\psi_t) \leq (n+1)^{-1} \int_X (\psi_t-V_{\theta}) \,\theta_{\psi_t}^n
$$
so our claim is equivalent to showing that
\[
\psi \in {\mathcal E}(X,\theta) \Longleftrightarrow t^{-1} \int_X (\psi_t-V_{\theta}) \theta_{\psi_t}^n\rightarrow 0.
\]
Fix $s\in (0,1)$. Note that on $\{\psi>V_{\theta}-t\}$, the two measures $\theta_{\psi_t}^n$ and $\theta_{\psi}^n$ coincide. Additionally, using that $X = \{\psi \leq V_\theta - t \} \cup \{V_\theta -t < \psi \leq V_\theta - st \} \cup \{V_\theta - st < \psi \}$ we can estimate the right-hand side above as follows: 
\[
0\geq \int_X \frac{\psi_t-V_{\theta}}{t} \theta_{\psi_t}^n \geq  -\int_{\{\psi\leq V_{\theta}-t\}} \theta_{\psi_t}^n - \int_{\{\psi\leq  V_{\theta}-ts\}} \theta_{\psi}^n -s\vol(\theta).
\]
By \cite[Lemma 1.2]{GZ07}, $\psi \in \mathcal E(X,\theta)$ if and only if $  \int_{\{\psi \leq V_\theta-t\}} \theta_{\psi_t}^n\rightarrow 0$ as ${t\rightarrow +\infty}$. Hence, the above estimates give the conclusion.
\end{proof}

\begin{lemma}
        \label{lem: AM along v} Suppose $\psi \in \psh(X,\theta)$ satisfies $\psi \leq V_\theta$. Let $t \to v(V_\theta,\psi)$ be the geodesic ray constructed in \eqref{eq: v_t_def}. Then 
        \begin{equation}
                \label{eq: AM along v}
                \AM(v(V_\theta,\psi)_t) =  t c_\psi.
        \end{equation}
        In particular,  $t \to v(V_{\theta},\psi)_t$ is constant if and only if $\psi\in \Ec(X,\theta)$.
\end{lemma}
\begin{proof} We go back to the construction of $t \to v(V_\theta,\psi)_t$ in \eqref{eq: v_t_def}. By Theorem \ref{thm: AM linear big}, for each $l$ fixed, the curves $t \to \AM(u^l_t)$ are linear hence we can write:
        \[
        \AM(u^{l}_t)=  \AM(V_\theta)+ \frac{t}{l}(\AM(\max(\psi,V_{\theta}-l))-\AM(V_\theta))=\frac{t}{l}(\AM(\max(\psi,V_{\theta}-l))-\AM(V_\theta)).
        \]
        Letting $l\to +\infty$, by \cite[Proposition 3.3]{BB10} we obtain \eqref{eq: AM along v}. If $t \to v(V_\theta,\psi)_t$ is constant equal to $V_\theta$, it follows that $c_\psi=0$, thus by Proposition \ref{prop: E_char_c_psi} we get $\psi \in \Ec(X,\theta)$. Conversely, if $\psi\in \Ec(X,\theta)$ then Proposition \ref{prop: E_char_c_psi} yields that $c_{\psi}=0$, hence $\AM$ is constant along $v(V_{\theta},\psi)$, thus $v(V_{\theta},\psi)$ is constant.
\end{proof}

\begin{remark}
To show that if $t \to v(V_{\theta},\psi)_t$ is constant then $\psi\in \Ec(X,\theta)$, we only need to prove the estimate $\AM(v(V_\theta,\psi)_t) \leq  t c_\psi$. This last inequality is a simple consequence of the convexity of the Monge--Amp\`ere energy (Theorem \ref{thm: AM convex big}) and the construction of the ray $t \to v(V_{\theta},\psi)_t$ \eqref{eq: v_t_def}.
\end{remark}

\begin{lemma}
        \label{lem: Legendre transform}Given a weak geodesic ray $[0,+\infty)\ni t \to \phi_t  \in \textup{PSH}(X,\theta)$, its Legendre transform $ \Bbb R \ni \tau \to \phi^*_\tau = \inf_{t \in [0,+\infty)}(\phi_t - t\tau)$
satisfies
$$\phi^*_\tau = P_\theta(\phi^*_\tau+C,\phi_0), \ \tau \in \Bbb R,C>0.$$
In particular, $P_{[\theta,\phi^*_\tau]}(\phi_0)=\phi^*_\tau$.
\end{lemma}
\begin{proof}One can repeat the argument in \cite[Theorem 5.3]{Dar13}. Fix $\tau \in \Bbb R$. The fact that $\phi^*_\tau$ is $\theta$-psh follows from Kiselman's minimum principle \cite{Kis78}. Suppose that $\phi^*_\tau \neq -\infty$ and fix $C > 0$. Since $\phi^*_\tau \leq \phi_0$, it results that $P_\theta(\phi^*_\tau+C,\phi_0) \geq \phi^*_\tau.$ Hence we only have to argue that:
$$P_\theta(\phi^*_\tau+C,\phi_0) \leq \phi^*_\tau.$$
Let $[0,1] \ni t \to g^l_t,h_t \in \textup{PSH}(X,\theta), \ l \geq 0$, be the weak geodesic segments defined by the formulas: $$g^l_t = \phi_{tl}-tl\tau,$$
$$h_t = P_\theta(\phi^*_\tau+C,\phi_0)-Ct.$$

Then we have $h_0 \leq \phi_0=\lim_{t \to 0}g^l_t=g^{l}_0$ and $h_1 \leq \phi^*_\tau \leq g^{l}_1$ for any $l \geq 0$. Hence, by definition of weak geodesics \eqref{eq: geoddef} we have
$$h_t \leq g^{l}_t, \ t \in [0,1],l \geq 0.$$

Taking the infimum in the above estimate over $l \in [0,+\infty)$ and then taking the supremum over $t \in [0,1]$,  we obtain:
$$P_\theta(\phi^*_\tau+C,\phi_0) \leq \phi^*_\tau.$$
Letting $C \to +\infty$ we obtain the last statement of the proposition.
\end{proof}

\subsection*{Proof of Theorem \ref{thm: characterization E big classes}}
 Theorem \ref{thm: characterization E big classes} is a consequence of the following result.

\begin{theorem}
        \label{thm: characterization E big full version}
        Let $\theta$ be a smooth closed  $(1,1)$-form whose cohomology class is big. For any $\psi\in \psh(X,\theta)$, the following are equivalent\\
        (i) $\psi\in \Ec(X,\theta)$;\\
        (ii) $P_{[\theta,\psi]}(\varphi)\in \Ec(X,\theta)$ for all $\varphi\in \Ec(X,\theta)$;\\
        (iii) $P_{[\theta,\psi]}(\varphi)=\varphi$ for all $\varphi\in \Ec(X,\theta)$;\\
        (iv) $P_{[\theta,\psi]}(V_{\theta})\in \Ec(X,\theta)$;\\
        (v)  $P_{[\theta,\psi]}(V_{\theta})=V_{\theta}$.
\end{theorem}
\begin{proof}
We can assume without loss of generality that $\psi\leq 0$. 
The implication $(i)\Longrightarrow (iii)$ follows from Theorem \ref{thm: Dravas criterion 1} while the implications $(iii)\Longrightarrow (ii) \Longrightarrow (iv)$ and $ (iii) \Longrightarrow (v)$ are trivial. 

We now prove that $(v)\Longrightarrow (i)$. Suppose that  $P_{[\theta,\psi]}(V_\theta)=V_{\theta}$. From the construction of the ray $t \to v(V_\theta,\psi)_t$ in \eqref{eq: v_t_def} it automatically follows that $v(V_\theta,\psi)_t \geq \psi$. This trivially gives $v^*_0 = \inf_{t \in [0,\infty) }v(V_\theta,\psi)_t \geq \psi$. By definition of envelope we have $V_\theta \geq P_{[\theta, v^*_0]}(V_\theta) \geq P_{[\theta, \psi]}(V_\theta)=V_\theta$. Combining this with Lemma \ref{lem: Legendre transform} we obtain that $v^*_0 = P_{[\theta, v^*_0](V_\theta)}=V_\theta$. As the ray $t \to v(V_\theta,\psi)_t$ is decreasing in $t$, this automatically gives that $V_\theta =v(V_\theta,\psi)_0\geq v(V_\theta,\psi)_t \geq v^*_0 = V_\theta$, hence $t \to v(V_\theta,\psi)_t $ is constant. Invoking Lemma \ref{lem: AM along v} we obtain that $\psi \in \mathcal E(X,\theta)$. 

It remains to prove that $(iv)\Longrightarrow (v)$. Assume that $\phi=P_{[\theta,\psi]}(V_{\theta})\in \Ec(X,\theta)$. By Lemma \ref{lem: AM along v} we have that $v(V_{\theta},\phi)_t=V_{\theta}$ for all $t$. Set  $v_t:=v(V_{\theta},\psi)_t$ as constructed in  \eqref{eq: v_t_def}. It follows from Lemma \ref{lem: Legendre transform} that $v_{\infty}= P_{[\theta,v_{\infty}]}(V_{\theta})\geq P_{[\theta,\psi]}(V_{\theta})$, hence $v_{\infty}\in \Ec(X,\theta)$. Hence  by Lemma \ref{lem: AM along v} we have that $v(V_{\theta},v_{\infty})$ is constant. Since $v_t$ is decreasing in $t$, it follows that $v_t\geq \max(V_{\theta}-t,v_{\infty}),\forall t\geq 0$. It thus follows  that $v_t=v(V_{\theta},v_{\infty})_t=V_{\theta}$, hence again  by Lemma \ref{lem: AM along v} we have that  $\psi\in \Ec(X,\theta)$.
\end{proof}

Theorem \ref{thm: sum_FULL_MASS_1} follows directly from Theorem \ref{thm: characterization E big classes}.

\begin{theorem}\label{thm: sum_FULL_MASS}
Let $\{ \theta_1\}, \{ \theta_2\}$ be big classes. The following are equivalent:\\
(i) $V_{\theta_1}+V_{\theta_2} \in \mathcal E(X, \theta_1 + \theta_2)$;\\
(ii) For any $u \in \psh(X,\theta_1), v \in \psh(X,\theta_2)$ we have
\[
u+v \in \mathcal E(X,\theta_1 + \theta_2) \Longleftrightarrow u\in \Ec(X,\theta_1), v\in \Ec(X,\theta_2).
\]      
\end{theorem}
\begin{proof}
Since $V_{\theta_j}\in \Ec(X,\theta_j), j=1,2$, the implication $(ii)\Longrightarrow (i)$ is trivial.
Assume (i) holds. The implication $(\Longrightarrow)$ in (ii) follows from \cite[Theorem B]{DN15}. Assume that $\varphi_{j}\in \Ec(X,\theta_j), j=1,2$. We want to show that $\varphi:=\varphi_1+\varphi_2 \in \Ec(X,\theta_1+\theta_2)$. By assumption that (i) holds we get that $V_{\theta_1}+V_{\theta_2}\in \Ec(X, \theta_1+\theta_2)$. Hence, by definition of envelopes we can write 
\[
P_{[\theta_1 +\theta_2,\varphi]}(V_{\theta_1}+V_{\theta_2}) \geq P_{[\theta_1,u]}(V_{\theta_1}) +P_{[\theta_2,v]}(V_{\theta_2})=V_{\theta_1}+V_{\theta_2}. 
\]
Ultimately, it follows from Theorem \ref{thm: characterization E big classes} that $\varphi \in \Ec(X,\theta_1+\theta_2)$.
\end{proof}

\section{Proof of Theorem \ref{thm: lelong number big class}} \label{sect: proof of Theorem 1.1}

In this section we give the proof of Theorem \ref{thm: lelong number big class} and discuss some immediate consequences.

\begin{proof}[Proof of Theorem \ref{thm: lelong number big class}] We first argue the equality of Lelong numbers in (i). From Theorem \ref{thm: characterization E big classes} it follows that $P_{[\theta,\varphi]}(V_{\theta})=V_{\theta}$.  Take any $x\in X$. Then trivially $\nu(\varphi,x) \geq \nu(V_{\theta},x)$. We will argue by contradiction. Assume that $\nu(\varphi,x) > \nu(V_{\theta},x)$. Fix a holomorphic coordinate around $x$ so that we identify $x$ with $0\in \BB \subset \CC^n$ where $B$ is the unit ball in $\CC^n$. By definition of the Lelong numbers \eqref{eq: Lelongdef} we have
\[\varphi(z) \leq \gamma \log \|z\| + O(1),\]
where $\gamma=\nu(\varphi,x)>0$. Let $g$ be a smooth local potential for $\theta$ in $\BB$ and observe that if $\psi \in \psh(X,\theta)$ then $g+\psi$ is psh in $\BB$. Furtheremore, w.l.o.g. we can assume that $g + \varphi,g + V_\theta \leq 0$ in $\BB$. By the definition of the envelope we have the following inequality
\[
V_\theta + g = P_{[\theta,\varphi]}(V_\theta) + g  \leq  \sup\{v\in \psh(\BB) \setdef v\leq 0\,,\;  v\leq \gamma \log \|z\| +O(1)\}, 
\]
in $\BB$. The right-hand side is the pluricomplex Green function $G_{\BB}(z,0)$ of $\BB$ with a logarithmic pole at $0$ of order $\gamma$. By \cite[Proposition 6.1]{Kli} we have that 
\[
G_{\BB}(z,0)   \sim \gamma \log \|z\| + O(1). 
\]
But this contradicts with the assumption that $\nu(V_\theta,x) < \gamma$. 

Now we argue the equality of multiplier ideal sheafs in (i). This will be an application of Theorem \ref{thm: characterization E big classes} and the resolution of the strong openess conjecture of Guan and Zhou \cite{GuZh}, in the specific form provided by Lempert \cite{Lemp}. Indeed, from Theorem \ref{thm: characterization E big classes} it follows that 
\begin{equation}\label{eq: thm12appl}
P_\theta(\varphi+c,V_\theta)(x) \nearrow V_\theta(x) \textup{ as } c \to \infty, \textup{ for a.e. } x \in X.
\end{equation}
As $\varphi \leq V_\theta + c'$ for some $c' \in \Bbb R$, we note that $\varphi$ and $P_\theta(\varphi+c,V_\theta)$ have the same singularity type for any $c\in \Bbb R$, ultimately giving $\mathcal I(tP_\theta(\varphi+c,V_\theta),x)=\mathcal I(t\varphi,x), \ x \in X, t\geq 0$. 

Finally, \eqref{eq: thm12appl} and \cite[Theorem 1.1]{Lemp} imply that $\mathcal I(tP_\theta(\varphi+c, V_\theta),x)=\mathcal I(tV_{\theta},x)$ for large enough $c$,  proving that $\mathcal I(t\varphi,x)=\mathcal I(tV_\theta,x)$.

Now we turn to part (ii). Fix $\omega$ a K\"ahler form on $X$. We can suppose that  $\theta, \eta \leq  \tilde \omega:=\eta + \omega$ and $\tilde \omega$ is K\"ahler. Assume that $\varphi\in \Ec(X,\eta) \cap \psh(X,\theta)$. By Theorem \ref{thm: characterization E big classes} we get that $P_{[\eta,\varphi]}(V_{\eta})=V_{\eta}$. This implies $P_{[\tilde \omega,\varphi]}(V_{\eta})=V_{\eta}$ since $P_{[\eta,\varphi]}(V_{\eta})\leq P_{[\tilde \omega,\varphi]}(V_{\eta})\leq V_{\eta}$. 

Furthermore, we claim that $V_{\eta}\in \Ec(X,\tilde \omega) $, i.e., $\int_X  \tilde \omega_{V_{\eta}}^n= \vol(\tilde \omega).$
Indeed, as $\theta$ is nef,  expanding the sum of K\"ahler classes $(\eta + (1 + \vep )\omega)^n$ gives 
\[
\vol(\{\eta + (1+\vep)\omega\})^n = \sum_{k=0}^n \binom{n}{k} \{\eta + \vep \omega\}^k\cdot \{\omega\}^{n-k}. 
\]
It follows from the comments after \cite[Definition 1.17]{BEGZ10} and \cite[Proposition 2.9]{BFJ09} that the left-hand side converges to $\vol(\tilde \omega)$ while the right-hand side converges to $\sum_{k=0}^n \binom{n}{k} \{\eta\}^k\cdot \{\omega\}^{n-k}$, ultimately giving 
\[
\vol(\{\tilde \omega\})=\vol(\{\eta+ \omega\}) = \sum_{k=0}^n \binom{n}{k} \{\eta\}^k\cdot \{\omega\}^{n-k}. 
\] 
On the other hand, by multilinearity of the non-pluripolar product we get
\begin{eqnarray*}
\int_X \tilde \omega_{V_{\eta}}^n = \int_X (\eta+\omega+dd^c V_{\eta})^n = \sum_{k=0}^n  \binom{n}{k}\int_X (\eta+ dd^c V_{\eta})^{k}\wedge \omega^{n-k},
\end{eqnarray*}
and moreover $\{(\eta+ dd^c V_{\eta})^{k}\} = \{\eta\}^{k} $ for each $0\leq k\leq  n-1$ thanks to \cite[Definition 1.17]{BEGZ10}, proving the claim.

Given that $P_{[\tilde \omega,\varphi]}(V_\eta)=V_\eta$ and $V_\eta \in \mathcal E(X,\tilde \omega),$ we can use \cite[Theorem 4]{Dar14a} to conclude that $\varphi\in \Ec(X,\tilde \omega)$. Because $\theta \leq \tilde \omega$ and $\varphi \in \psh(X,\theta)$, we get $\varphi\in \Ec(X,\theta)$, as follows from \cite[Theorem B]{DN15}.
\end{proof}


\begin{remark}\label{NonFullMass}
Observe that Theorem \ref{thm: lelong number big class} (ii) is in general false for classes $\{\eta\}$ that are big but not nef. Indeed, if $\{\eta\}$ is only big, it may happen that $V_{\eta}$ has a non-zero Lelong number at some point, and then \cite[Corollary 2.18]{GZ07} would give us that $V_{\eta}$ does not have full mass with respect to any K\"ahler class $\{\theta\}$ satisfying $\eta \leq \theta$, contradicting $\mathcal E(X,\eta)\cap \textup{PSH}(X,\theta) \subset \mathcal E(X,\theta)$.
\end{remark}

As a direct consequence we obtain the following additivity property of the set of full mass currents of big and nef cohomology classes, effectively proving that condition (i) in Theorem \ref{thm: sum_FULL_MASS_1} is automatically satisfied. 

\begin{coro}\label{thm: sum_FULL_MASS_nef}
Let $\{\theta_1\},\{ \theta_2\}$ be big and nef classes. Then for any $\varphi_1 \in \textup{PSH}(X,\theta_1)$ and $\varphi_2 \in \textup{PSH}(X,\theta_2)$ we have
\[\varphi_1 + \varphi_2 \in \Ec(X,\theta_1 + \theta_2) \Longleftrightarrow \varphi_1 \in \Ec(X,\theta_1), \varphi_2 \in \Ec(X,\theta_2). 
\]      
\end{coro}

\begin{proof} Fix a K\"ahler form $\omega$ such that $\theta_{j}\leq \omega, j=1,2$. It follows from part (ii) of Theorem \ref{thm: lelong number big class}  that $\varphi_j\in \Ec(X,\omega), \forall j=1,2$. By the convexity of $\Ec(X,\omega)$ proved in \cite[Proposition 1.6]{GZ07} it follows that $\varphi_1+\varphi_2$ belongs to  $\Ec(X,2\omega)$. Now, \cite[Theorem B]{DN15} gives that $\varphi_1+\varphi_2\in \Ec(X,\theta_1+\theta_2)$, hence the result follows. 
\end{proof}

\section{Further Applications}\label{sect: application}
\subsection{Invariance of finite energy classes}
The following result says that finite energy classes are invariant under bimeromorphic maps as soon as the volume is preserved. The result was recently obtained in \cite{DNFT}. As an application of Theorem \ref{thm: lelong number big class} we give a slightly different proof of the "baby case", i.e. when $f$ is a blow-up along a smooth center. 
 
\begin{prop}\label{MA_energy}
Let $\pi: X\rightarrow Y$ be a blow up with smooth center $\mathcal{Z}$ between K\"ahler manifolds and $E$ be the exceptional divisor. Let $\alpha\in H^{1,1}(X, \RR)$ be a big class. Then the following conditions are equivalent:
\begin{itemize}
\item[(i)]$ \vol(\alpha)= \vol(\pi_\star \alpha)$;
\item[(ii)] given a positive $(1,1)$-current $T$ in $\pi_\star \alpha$, then  $S=\pi^\star T+\gamma [E]$, where $\gamma$ is a cohomological factor, is a positive $(1,1)$-current on $X$.
\item[(iii)] $\pi_{\star}(\Ec(X,\theta))= \Ec(Y,f_{\star}\theta)$,
\item[(iv)]  $\pi_{\star}(\Ec_{\chi}(X,\theta))= \Ec_{\chi}(Y,f_{\star}\theta)$ for any weight $\chi$.
\end{itemize}
\end{prop}
\begin{proof}

Recall that, given a smooth representative $\theta$ of the class $\alpha$, it follows from $\partial\bar{\partial}$-lemma that any positive $(1,1)$-current can be written as $T=\theta+dd^c \varphi$ where the global potential $\varphi$ is a $\theta$-psh function, i.e. $\theta+dd^c \varphi\geq 0$.
The implications $(iii)\Rightarrow (i)$ and  $(iv)\Rightarrow (i)$ are trivial while the fact that $(ii)$ implies $(i)$,$(iii)$ and $(iv)$ are \cite[Proposition 3.3]{DN15}. We want to prove $(i)\Rightarrow (ii)$. it suffices to show that for any positive $(1,1)$-current $T$ we have $\gamma\geq - \nu(T, \mathcal{Z})$.
Let $S_{\min}$ be a positive  current with minimal singularities in $\alpha$, then it writes as
$$
S_{\min}= \pi^\star T_{Y}+ \gamma[E]
$$ where $\gamma\geq-\nu(T_Y, \mathcal{Z})$. It easy to check that $T_Y\in \Ec(Y, \pi_\star \alpha)$. Indeed,
$$\vol(\pi_\star \alpha)=\vol(\alpha)= \int_X \langle S_{\min}^n\rangle= \int_Y\langle T_Y^n \rangle. $$
Thus, it follows from Theorem \ref{thm: lelong number big class} that $\nu(T_Y, y)= \nu(T_{\min}, y)$ for any $y\in Y$ and for any $T_{\min}$ current with minimal singularities on $Y$. Hence $\nu(T_{\min},\mathcal{Z})+\gamma\geq 0$. Furthermore any positive current $T\in \pi_\star \alpha$ is such that $\nu(T_{\min},\mathcal{Z})\leq \nu(T, \mathcal{Z})$, thus $\nu(T,\mathcal{Z})+\gamma\geq 0$.
\end{proof}

\subsection{Log concavity of non-pluripolar product}
It was conjectured in \cite[Conjecture 1.23]{BEGZ10} that 
\begin{equation}\label{eq: log concave}
        \int_X \langle T_1\wedge ...\wedge T_n\rangle \geq \left(\int_X \langle T_1^n\rangle \right)^{1/n} ...\left(\int_X \langle T_n^n\rangle \right)^{1/n},
\end{equation}
for all positive currents $T_1,...,T_n$. The result holds for currents with analytic singularities as mentioned in \cite{BEGZ10}. In this subsection we  confirm this conjecture in the case of full mass currents in big and nef classes. 

With the help of Corollary \ref{thm: sum_FULL_MASS_nef}, we can make obvious adjustments in the proof of \cite[Corollary 2.15]{BEGZ10} to get the following result:
\begin{prop}
        \label{cor: mixed volume}
        Let $\alpha_i$, $i=1, \cdots , n$ be  big and nef cohomology classes and let  $T_i\in \mathcal{E}(X, \alpha_i)$.  Then we have
\[
\int_X \langle T_1\wedge \cdots \wedge T_n\rangle =\int_X \langle T_{1,\min}\wedge \cdots \wedge T_{n,min}\rangle.
\]
\end{prop}

Using this  we can prove the log concavity of full mass currents in the big and nef case:

\begin{coro}\label{cor: log concave big class}
        If $T_j,j=1,\cdots n$,  are full mass currents in big and nef cohomology classes then  \eqref{eq: log concave} holds.
\end{coro}
\begin{proof}
Let $\mu$ denote the non-pluripolar measure $\mu:=\langle T_1\wedge ...\wedge T_n\rangle$ and let $\lambda_j, j=1,\cdots ,n$ be positive constants such that 
\begin{equation}\label{eq: lambda_def}
\lambda_j\mu(X) =\int_X \langle T_j^n \rangle. 
\end{equation}
For each $j$, using \cite[Theorem A]{BEGZ10} there exists a positive full mass current $S_j\in \{T_j\}$ such that 
$$ \langle S_j^n \rangle = \lambda_j \mu. $$
By \cite[Proposition 1.11]{BEGZ10} we have that 
\[
\langle S_1\wedge ...\wedge S_n\rangle \geq (\lambda_1\cdots \lambda_n)^{1/n} \mu. 
\]
Proposition \ref{cor: mixed volume} gives that $\int_X \langle S_1\wedge ...\wedge S_n\rangle=\int_X \langle T_1\wedge ...\wedge T_n\rangle$. Hence after integrating the above inequality, due to \eqref{eq: lambda_def}, the result follows.
\end{proof}

\paragraph{Acknowledgments. } The  first named author has been partially supported by BSF grant 2012236 and NSF grant DMS--1610202. The second named author is supported by a Marie Sklodowska Curie individual fellowship 660940--KRF--CY (MSCA--IF). The last named author has been supported by the SNS project  "Pluripotential theory in differential geometry". 
 
The last named author would like to thank D. Angella and S. Calamai for organizing a Ph.D course at University of Florence, which inspired part of this work. At the time this paper was written, the first and second named authors were in residence at the Mathematical Sciences Research Institute in Berkeley, Cali\-fornia, attending the "Differential Geometry" thematic semester, and were supported by the NSF grant DMS--1440140.
 
 We thank V. Tosatti for reading an initial version of this article, generously sharing some of his ideas, and pointing out the connection with multiplier ideal sheaves that we ultimately integrated into Theorem \ref{thm: lelong number big class}. We would like to thank V. Guedj for his suggestions on how to improve the paper. Lastly, we would also like to thank the anonymous referee for requesting to clarify the relationship of our arguments with the regularity results of \cite{BD12}.

\let\OLDthebibliography\thebibliography 
\renewcommand\thebibliography[1]{
  \OLDthebibliography{#1}
  \setlength{\parskip}{1pt}
  \setlength{\itemsep}{1pt plus 0.3ex}
}

\noindent{\sc University of Maryland}\\
{\tt tdarvas@math.umd.edu}\vspace{0.1in}\\
\noindent{\sc Imperial College London}\\
{\tt e.di-nezza@imperial.ac.uk}\vspace{0.1in}\\
{\sc Scuola Normale Superiore, Pisa, Italy}\\
{\tt chinh.lu@sns.it}
\end{document}